\documentclass[11pt]{article}
\usepackage[top=1.1in, left=1in, right=1in, bottom=1.1in]{geometry}
\usepackage{authblk}
\usepackage{bbm}
\usepackage[dvipsnames]{xcolor}
\usepackage{soul}
\usepackage{subcaption}
\usepackage{amssymb}
\usepackage{amsthm}
\usepackage{amsmath}
\usepackage{multicol}
\usepackage{mathrsfs}
\usepackage{graphicx}
\usepackage{enumerate}
\usepackage{enumitem}
\usepackage{bm}
\usepackage{setspace}
\usepackage{verbatim}
\usepackage[export]{adjustbox}
\usepackage{multicol}
\usepackage{bbm}
\usepackage{lipsum}
\usepackage[ruled, vlined, linesnumbered]{algorithm2e}
\SetKwProg{custom}{}{}{}

\let\oldnl\nl
\newcommand{\nonl}{\renewcommand{\nl}{\let\nl\oldnl}}
\usepackage{bm}
\usepackage{enumitem}
\usepackage[]{algorithm2e}
\usepackage{lipsum}
\usepackage[utf8]{inputenc}
\usepackage{chngcntr}
\usepackage{apptools}
\AtAppendix{\counterwithin{lemma}{section}}

\usepackage{hyperref}
\usepackage{xfrac}
\usepackage{epstopdf}
\usepackage{threeparttable}
\usepackage{booktabs}
\usepackage{mathtools}

\SetKw{KwInitialization}{Initialization}
\SetKw{KwTheMainBody}{The main body}

\newtheorem{lemma}{Lemma}[section]

\newtheorem{corollary}{Corollary}[section]
\newtheorem{assumption}{Assumption}
\newtheorem{theorem}{Theorem}[section]

\newtheorem{definition}[theorem]{ Definition}

\newtheorem{remark}{Remark}[section]

\definecolor{red}{rgb}{1,0.2,0.2}

\newcommand{\R}{\mathbb R}

\DeclareMathOperator{\E}{\mathbb{E}}

\def\R{\mathbb{R}}

\usepackage{ulem}
\newcommand\Lukesout{\bgroup\markoverwith{\textcolor{blue}{\rule[0.5ex]{2pt}{0.4pt}}}\ULon}

\title{Computing committors in collective variables \\via Mahalanobis diffusion maps}%
\author[1]{Luke Evans\thanks{evansal@umd.edu}}
\author[1]{Maria K. Cameron\thanks{mariakc@umd.edu}}
\author[2]{Pratyush Tiwary\thanks{ptiwary@umd.edu}}
\affil[1]{\small{Department of Mathematics, University of Maryland, College Park, MD 20742, USA}}
\affil[2]{\small{Department of Chemistry and Biochemistry, University of Maryland, College Park, MD 20742, USA}}                               

\begin{document}

\maketitle

\abstract{
The study of rare events in molecular and atomic systems such as conformal changes and 
cluster rearrangements has been one of the most important research themes in chemical physics. 
Key challenges are associated with long waiting times rendering molecular simulations inefficient, 
high dimensionality impeding the use of PDE-based approaches, 
and the complexity or breadth of transition processes limiting the predictive power of asymptotic methods. 
Diffusion maps are promising algorithms to avoid or mitigate all these issues. 
We adapt the diffusion map with Mahalanobis kernel proposed by Singer and Coifman (2008) 
for the SDE describing molecular dynamics in collective variables in which the 
diffusion matrix is position-dependent and, unlike the case considered by Singer and Coifman, 
is not associated with a diffeomorphism. We offer 
an elementary proof showing that one can approximate the generator for this 
SDE discretized to a point cloud via the Mahalanobis diffusion map. 
We use it to calculate the 
committor functions in collective variables for two benchmark systems: alanine dipeptide, and Lennard-Jones-7 in 2D. 
For validating our committor results, we compare our committor functions to the finite-difference solution or
by conducting a ``committor analysis" as used by molecular dynamics practitioners. We contrast the outputs of the 
Mahalanobis diffusion map with those of the standard diffusion map with isotropic kernel and 
show  that the former gives significantly more 
accurate estimates for the committors than the latter.}
\section{Introduction}
\label{sec:intro}
{ \subsection{Motivation}}
\label{sec:motivation}
Molecular simulation commonly deals with high-dimensional systems that reside in stable
states over very large timescales and transition quickly between these states on extremely
small scales. These transitions, rare events such as protein folding or
conformational changes in a molecule, are crucial to molecular simulations
but difficult to characterize due to the timescale gap. 
Transition path theory (TPT) is a mathematical framework for direct study of rare transitions
in stochastic systems, and it is particularly utilized for metastable systems
arising in molecular dynamics (MD)~\cite{weinan2004}. The key function of TPT is
the committor function, a mathematically well-defined reaction coordinate with
which one can compute reaction channels and expected transition times between a
reactant state $A$ and a product state $B$ in the state space. 
However,
the committor is the solution to an elliptic PDE that can be solved
using finite difference or finite element methods only in low dimensions.
As a result, the typical use of
transition path theory for high-dimensional systems consists of finding an estimate for 
a zero-temperature asymptotic transition path and then using techniques like 
umbrella sampling to access the committor~\cite{weinan2006towards,EVE2010}.
While this approach is viable, it relies on the assumption that the transition process is localized to a narrow tube around the found path. 
In practice, the transition process may be broad and complex, and the found asymptotic path may give a poor prediction.

%

One can utilize intrinsic dimensionality of the system that is typically much lower than 3$N_a$, where $N_a$ is the number of atoms, 
and analyze
transition paths in terms of {\it collective variables}~\cite{cameron2013estimation,
2006string}. However, a large number of internal coordinates such as contact distances and dihedral angles 
may be required for a proper representation of a biomolecule~\cite{ravindra2020automatic,rohrdanz2011determination, piana2008advillin,sittel2018perspective}. 
Hence, one still may be unable to leverage traditional mesh-based PDE solvers to even dimensionally-reduced data given either in physics-informed or machine-learned collective variables.

{ \subsection{An overview}}
\label{sec:overview}
Meshless approaches to solving the committor PDE discretize it to point clouds obtained from MD simulations.
The two most promising approaches utilize neural networks \cite{khoo2019solving,ren2019,rotskoff2020} and/or diffusion maps~\cite{trstanova2020local}.
The approach based on neural networks is more straightforward in its implementation, while the one based on diffusion maps is more interpretable, visual, 
and intuitive, and we will focus on it in this work. 
We also would like to acknowledge the work of Lai and Lu (2018) \cite{lai2018point} on computing committors discretized to point clouds.
They proposed and advocated the ``local mesh" algorithm and highlighted its advantage over an approach utilizing diffusion maps. Contrary to diffusion maps, the ``local mesh" does not require the input data to be sampled from an invariant distribution, which is indeed a very appealing feature. However, the ``local mesh''
implementation only considers SDEs with constant noise and has no theoretical guarantees of convergence, while we found diffusion maps simple, robust, reliable, and deserving further development.

The diffusion map
algorithm introduced in 2006 in the seminal work by Coifman and Lafon~\cite{coifman2006diffusion} is a widely used
manifold learning algorithm. 
Like its predecessors such as  
locally linear embedding~\cite{roweis2000nonlinear}, isomap~\cite{tenenbaum2000global}, and the Laplacian
eigenmap~\cite{belkin2003laplacian},
diffusion map relies on the assumption that the dataset lies in the vicinity of a certain low-dimensional manifold, 
while the dimension of the ambient space can be high. 
Importantly, this manifold does not need to be known a priori.
Diffusion map inherits the use of a kernel for learning the local geometry from its predecessors and 
upgrades it with the remarkable ability to approximate a class of differential operators  discretized to the dataset.
These include the backward Kolmogorov operator (a.k.a. the generator) needed for computing the committor function.

More specifically, the standard diffusion map with isotropic Gaussian kernel~\cite{coifman2006diffusion} yields
an approximation to the backward Kolmogorov operator if the input data 
comes from an Ito diffusion process with a gradient drift and an isotropic additive noise. The problem of finding the committor then reduces
to solving a system of linear algebraic equations of a manageable size. This strategy would be suitable for MD data in the original 
$\mathbb{R}^{3N_a}$-dimensional space of atomic positions, and has been utilized heavily in previous work~\cite{rohrdanz2011determination,trstanova2020local, ferguson2010systematic, ferguson2011nonlinear, zheng2011polymer, coifman2008diffusion}.
As mentioned above, biophysicists prefer to keep track of collective variables rather than positions of atoms, as it lowers the dimension and gives more useful   
information about the molecular configuration.
Unfortunately, the transformation to
collective variables induces anisotropy and position-dependence on the noise term~\cite{weinan2004, cameron2013estimation, 2006string, berez2014multidim}.
The  resulting SDE describing the dynamics in collective variables has a non-gradient drift and a multiplicative anisotropic noise: {
\begin{equation}
\label{eqn:cvsde}
dx_t = \left[-M(x_t)\nabla F(x_t) + \nabla\cdot M(x_t)\right]dt + \sqrt{2\beta^{-1}} M^{1/2}(x_t)dw_t.
\end{equation}
Here,  the \emph{diffusion matrix} $M(x)$, is a symmetric positive definite matrix function, $\nabla F(x)$ is the \emph{mean force}, the gradient of the \emph{free energy}, and $d w$ is the increment of the standard multidimensional Brownian motion. The definitions and the computation of $M(x)$ and $\nabla F(x)$ are detailed in \ref{sec:diffusion_estimation}.
}

{\subsection{The goal and a brief summary of main results}}
The goal of the present work is to extend  { the diffusion map algorithm for computing the committor functions  for data sampled from the invariant density of SDE \eqref{eqn:cvsde} and prove theoretical guarantees for its correctness.}
We emphasize that our application of diffusion maps is not for learning order parameters or doing dimensional reduction as in 
many previous works, e.g. \cite{rohrdanz2011determination, ferguson2010systematic, ferguson2011nonlinear,coifman2008diffusion}. We assume that we already have
a dimensionally reduced system due to the use of collective variables. Instead, we are going to approximate the 
generator of SDE \eqref{eqn:cvsde} by means of diffusion maps and use it to compute the committor.

If { the diffusion matrix $M(x)$ in SDE \eqref{eqn:cvsde}} arose from a diffeomorphism, we could straightforwardly apply the diffusion map with the Mahalanobis kernel introduced by Singer and Coifman in 2008~\cite{singer2008}. 
However, the transformation to collective variables is not a diffeomorphism.  
Typically, the number of collective variables is much smaller than $3N_a$ and the computation of collective variables involves an averaging with respect to the invariant probability density. 
As a result, all that can be guaranteed is that the diffusion matrix $M(x)$ in \eqref{eqn:cvsde} 
is symmetric positive definite. { Furthermore, it is known that not every symmetric positive definite matrix function $M(x)$ is decomposable to $M(x)=J(x)J^\top(x)$ 
where $J(x)$ is the Jacobian matrix for some diffeomorphism \cite{risken1996fokker}, as required for the formalism introduced in \cite{singer2008}.}

{ Our results are the following:
\begin{itemize}
\item We offer an elementary proof showing that a smooth symmetric positive definite $d\times d$ matrix function defined in an open set $\Omega$ is not necessarily decomposable as $JJ^\top(x)$ where $J(x)$ is the Jacobian matrix of some smooth vector-function $f(x)$ in $\Omega$. Moreover, we establish a criterion for the existence of such a decomposition for a special class of $2\times 2$ matrix functions of the form $M(x) = m^2(x)I_{2\times 2}$ where $m(x)$ is a nonnegative twice continuously differentiable function in a simply connected open set $\Omega\subset\mathbb{R}^2$.
    \item We prove a theorem establishing a family of differential operators approximated by the family of diffusion maps with Mahalanobis kernel with an arbitrary smooth symmetric positive definite matrix function $M(x)$ parametrized by the renormalization parameter $\alpha\in\mathbb{R}$. We also prove that if $\alpha = \sfrac{1}{2}$ the corresponding diffusion map with the Mahalanobis kernel approximates exactly the generator for SDE \eqref{eqn:cvsde}. Our proofs involve only elementary tools such as multivariable calculus and linear algebra. We will refer to the resulting diffusion map algorithm with the Mahalanobis kernel as {\tt mmap}. We will compare its results to the original diffusion map algorithm with isotropic Gaussian kernel and refer to it as {\tt dmap}.

\item We apply {\tt mmap} with $\alpha=\sfrac{1}{2}$ to two common test systems: the
alanine dipeptide molecule~\cite{rohrdanz2011determination, trstanova2020local, ferguson2010systematic, chen2018collective, stamati2010application, wang2021state}
and the Lennard-Jones cluster of 7
particles in 2 dimensions (LJ7)~\cite{dellago1998efficient, wales2002discrete, passerone2001action}. 
For both systems, we compute the committor function on trajectory data and provide validation for the results. 
For alanine dipeptide, we compare the committor obtained using {\tt mmap} with the one computed using 
a finite difference method and demonstrate good agreement between these two. 
We contrast the committor obtained from {\tt mmap} with the one obtained by {\tt dmap} and show that the latter is notably less accurate.
We also compute the reactive currents for both {\tt mmap} and {\tt dmap} and obtain an estimate for the transition rate.
In addition, we investigate the dependence of the error in the estimate for the committor on the scaling parameter 
$\epsilon$ in the kernel of {\tt mmap} and {\tt dmap} and conclude that {\tt mmap} is consistently 
more accurate and at least as robust as {\tt dmap}.
For LJ7, we conduct a committor analysis, a simulation-based validation technique for the committor~\cite{2006string,peters2016reaction,geissler1999kinetic}.
Our results indicate that {\tt mmap} produces a reasonable approximation for the 
$\sfrac{1}{2}$-isocommittor surface, while {\tt dmap} places this surface at an utterly wrong place in the collective variable space. 
\end{itemize}
}

The rest of the paper is organized as follows.
In Section~\ref{sec:background} we review relevant concepts from MD in collective variables, transition path
theory and diffusion
maps. Our theoretical results are  presented in 
Section~\ref{sec:main}.
Applications to alanine dipeptide and LJ7  are detailed in 
Section~\ref{sec:examples}. 
Concluding remarks are given in Section
\ref{sec:conclusion}. 
Various technical points and proofs are worked out in the appendices.


\section{Background}
\label{sec:background}
In this section, we give a quick overview on collective variables, transition path theory, and diffusion maps.

\subsection{Effective dynamics and collective variables} 
\label{sec:colvar}

Our primary interest in this work is in datasets arising in MD simulations.
We consider the \emph{overdamped Langevin equation}, a
simplified model for molecular motion which describes the molecular configuration in terms of the  
positions $y$ of its atoms:
\begin{equation}
\label{eqn:overdamped_sde}
dy_t = -\nabla V(y_t) dt + \sqrt{2 \beta^{-1}} dw_t, 
\end{equation}
where $y \in \R^{m} $, $V:\R^{m} \to \R$ is a potential
function, $\beta^{-1} = k_b T$ is temperature in units of
Boltzmann's constant, $t$ is time, and $w_t$ is a Brownian motion in
$\R^{m}$. 
Given certain conditions on the potential $V$, a system
governed by overdamped Langevin dynamics is ergodic with respect to the \emph{Gibbs
distribution} $\mu(y) = Z_y^{-1} e^{-\beta V(y)},$ where $Z_y$ is a normalizing
constant.  The overdamped Langevin dynamics has infinitesimal generator 
\begin{equation}
    \label{eqn:overdamped_generator}
    \mathcal{L}f = \beta^{-1}\Delta f - \nabla f \cdot \nabla V = \beta^{-1} e^{\beta V} \nabla \cdot (e^{-\beta V} \nabla f)
\end{equation}
defined for twice continuously differentiable, square-integrable functions $f$.

As mentioned in the introduction, the number of atoms in biomolecules is typically very large. Even for such a small molecule as alanine dipeptide,
the number of atoms is 22 and results in a 66-dimensional configuration space ($m=66$).
Furthermore, to describe the state of a biomolecule one does not need atomic positions $y$ per se but rather certain functions in $y$
specifying desired geometric characteristics. 
Therefore, to reduce the dimensionality and obtain a more useful and comprehensive description of the system-at-hand, one uses
 \emph{collective variables} (CVs).
CVs are functions of the atomic coordinates
designed to give a coarse-grained description of the system's dynamics,
preserving transitions between metastable states but erasing small-scale
vibrations. Physical intuition has traditionally driven the choice of
collective variables including
dihedral angles, intermolecular distances, macromolecular distances and
various experimental measurements.

We denote the set of CVs as the vector-valued function { $x = \theta(y).$}
Since our goal is to compute the committor  (its precise definition is given in Section \ref{sec:tpt}),  a chosen set of CVs is good if the committor is well-approximated by a function that depends only on $\theta(y)$.
The dynamics in collective variables is approximated
by a diffusion process governed by  \cite{2006string} SDE \eqref{eqn:cvsde}: 
\begin{equation*}
dx_t = \left[-M(x_t)\nabla F(x_t) + \beta^{-1} \nabla \cdot M(x_t)\right]dt + \sqrt{2\beta^{-1}}M(x_t)^{1/2} dw_t.
\end{equation*}
{ Here, $F(x)$ is the \emph{free energy} and $M(x)$ is the \emph{diffusion matrix}, a symmetric positive definite matrix function. They are computed by averaging the appropriate functions of $x=\theta(y)$ over all $y\in\mathbb{R}^m$ with respect to the invariant density $\mu(y) = Z_y^{-1} e^{-\beta V(y)}$. The exact formulas for $F(x)$ and $M(x)$ and their evaluation in practice are detailed in \ref{sec:diffusion_estimation}.}

The generator for SDE \eqref{eqn:cvsde} is given by 
\begin{equation}
\label{eqn:mgen}
\mathcal{L}f = (-M\nabla F + \beta^{-1}(\nabla \cdot M))^{\top}\nabla f + \beta^{-1}{\sf tr}[M\nabla \nabla f],
\end{equation}
which can also be written in divergence form as 
\begin{equation}\label{eqn:mgen_divergence}
\mathcal{L}f = \beta^{-1} e^{\beta F} \nabla \cdot(e^{-\beta F}M \nabla f).
\end{equation}

One can check \cite{weinan2006towards} that the process given by an SDE of the form \eqref{eqn:cvsde}
is reversible and the invariant probability measure for \eqref{eqn:cvsde} is $\rho(x)=Z^{-1} e^{-\beta F(x)}$, the Gibbs measure.
Moreover, we would like to remark that any reversible diffusion process must be of the form \eqref{eqn:cvsde} \cite{pavliotis2014stochastic}.


\subsection{Transition path theory in collective variables}
\label{sec:tpt}
Throughout this section we assume that the system under consideration is governed by the overdamped Langevin SDE in collective variables \eqref{eqn:cvsde}. Suppose we have an infinite trajectory $\{x_t\}_{t=0}^{\infty}$. Further, suppose that we have designated a
priori two minima $x_A,$ $x_B$ of the potential $F$ with corresponding
disjoint open subsets $A \ni x_A, B\ni x_B$ which we refer to as the \emph{reactant} and
\emph{product} sets respectively. 
\begin{figure}[htbp]
\begin{center}
    \includegraphics[width=0.7\textwidth]{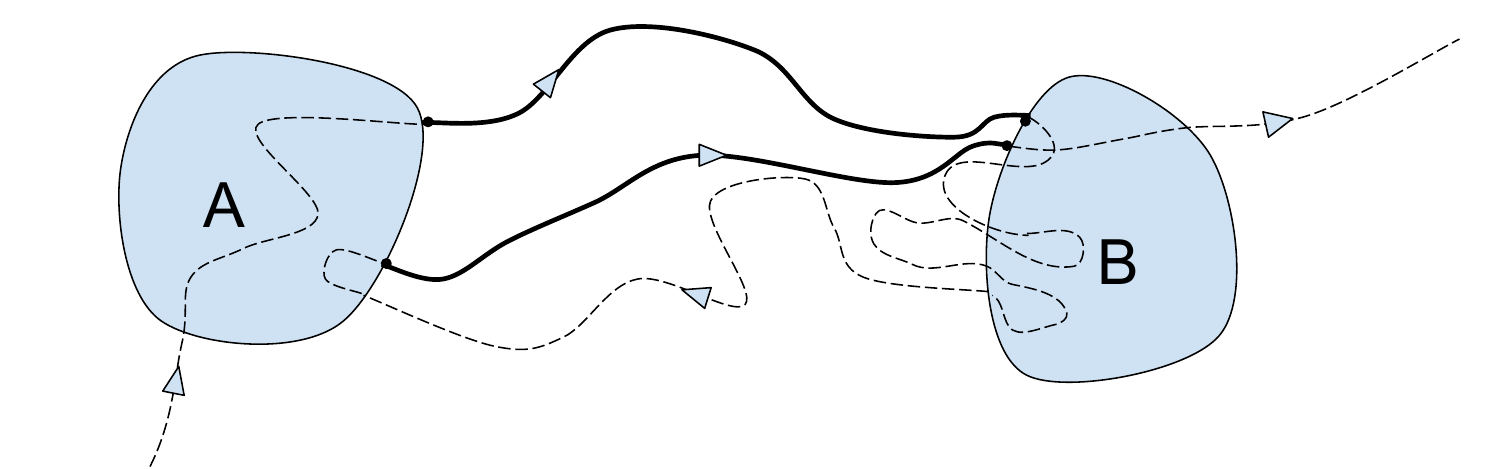}
\caption{
{\small A segment of a long trajectory. Reactive pieces from reactant state $A$ to product state $B$ are shown with solid lines.}
} 
\label{fig:reactive}
\end{center}
\end{figure}
Transition path theory (TPT)~\cite{weinan2006towards,EVE2010} is a mathematical framework to analyze statistics of transitions 
between the reactant $A$ and the product $B.$
The subject of TPT is the ensemble of \emph{reactive trajectories}, defined as any
continuous pieces of the trajectory $x_t$ which start at $\partial A$ and
end at $\partial B$ without returning to $\partial A$ in-between (see
Figure~\ref{fig:reactive}).
Key concepts of TPT are the
\emph{forward and backward committor functions} with
respect to $A$ and $B$. Since the governing SDE \eqref{eqn:cvsde} is reversible, 
the forward $q_{+}$ and backward $q_{-}$ committors are related via  $q_{-}=1-q_{+}$ \cite{weinan2006towards}.
Hence, for brevity, we will refer to the forward committor as \emph{the committor} and denote it merely by $q(x)$.
The committor $q$ has a straightforward probabilistic interpretation: 
\begin{equation}
\label{eqn:comm_def}
q(x) = \mathbb{P}(\tau_B < \tau_{A} \mid x_0=x),
\end{equation}
where $\tau_A:= \inf \lbrace t > 0 \mid x_t \in A \rbrace$ and $ \tau_B = \inf \lbrace t > 0 \mid x_t \in B \rbrace$ are
 the \emph{first entrance times} of the sets $A$ and $B$, respectively.
In words, $q(x)$ is the probability that a trajectory starting at $x$ will arrive at the product set $B$ before arriving at the reactant set $A$.
One can show that $q$ satisfies the boundary value problem~\cite{weinan2006towards}
\begin{equation}\label{eqn:comm_pde}
\begin{cases}
\mathcal{L}q(x) =  0 &  x \notin (A \cup B) \\
q(x) = 0 & x \in \partial A \\
q(x) = 1 & x \in \partial B, \\
\end{cases}
\end{equation}
where $\mathcal{L}$ is the infinitesimal generator~\eqref{eqn:mgen}.
Once the committor is computed, one can find the reactive current that reveals the mechanism of the transition process:
\begin{equation}
\label{eqn:cvrcur}
\mathcal{J} = \beta^{-1}Z^{-1}e^{-\beta F(x)}M(x)\nabla q(x).
\end{equation}
The integral of the flux of the reactive current through any hypersurface $\Sigma$
separating the sets $A$ and $B$ gives the reaction rate:
\begin{equation}
\label{eqn:nuAB}
\nu_{AB} :=\lim_{t\rightarrow\infty} \frac{N_{AB}}{t}= \int_{\Sigma}\mathcal{J}\cdot{\hat{n}}d\sigma,
\end{equation}
where $N_{AB}$ is the total number of transitions from $A$ to $B$ performed by the system 
within the time interval $[0,t]$, and $\hat{n}$ is the unit normal to the surface $\Sigma$ pointing in the direction of $B$.

Transition path theory has been extended to Markov jump processes~\cite{metzner2009,cameron2014flows}
on a finite state space $\mathcal{S}$, $|\mathcal{S}|=n$, defined by the generator matrix $L$ satisfying
\begin{equation}
\label{eqn:Ldiscrete}
\begin{cases}
\sum_{j\in \mathcal{S} } L_{ij} = 0,  &i\in\mathcal{S} \\
L_{ij} \ge 0, &i \neq j,~i,j\in\mathcal{S}.
\end{cases}
\end{equation}
The settings and concepts in discrete TPT mimic those from its continuous counterpart.
In particular, the committor is the vector $[q] = [q_1, \ldots, q_n ]^{\top}$ with
$$
[q]_i = \mathbb{P}(\tau_B < \tau_{A} \mid X_0 = i).
$$ 
Analogously, $[q]$ 
solves the matrix equation 
\begin{equation}\label{eqn:comm_matrix}
\begin{cases}
[Lq]_i =  0 &  i \in \mathcal{S}\backslash(A\cup B), \\
[q]_i = 0 & i \in A, \\
[q]_i = 1 & i \in B. 
\end{cases}
\end{equation}

Our goal is the following.
Let $\{x_i\}_{i=1}^{n}$ be a dataset sampled from SDE~\eqref{eqn:cvsde}.
We need to construct a \emph{discrete} generator $L$ such that, given a smooth scalar function $f(x)$, we have
$$
\sum_{j=1}^n L_{ij} f(x_j) \approx\mathcal{L}f(x_i)\quad \text{for each}\quad x_i,~1\le i\le n,
$$  
where $\mathcal{L}$ is the generator of~\eqref{eqn:cvsde} defined in~\eqref{eqn:mgen}.
We refer to this approximation as a
\emph{pointwise approximation of the generator} with respect to the dataset.
Given a pointwise approximate generator matrix $L$, we can then pointwise approximate the continuous
committor $q$ via our solution to the discrete committor equation~\eqref{eqn:comm_matrix}. 
In this work, we will show that the Mahalanobis diffusion map ({\tt mmap}) yields the desired approximation.


\subsection{Diffusion Maps}
\label{sec:diffusion_maps}
The diffusion map algorithm takes as input a dataset $X = \{x_i\}_{i=1}^n \subset \R^{d}$
of independent samples drawn from a distribution $\rho(x)$ that does not need to be known in advance. 
The manifold learning framework assumes that $X$ lies near a manifold $\mathcal{M}$ which has low intrinsic dimension.
Pairwise similarity of data is encoded through a kernel function
$k_{\epsilon}(x,y)$ whose simplest form is given by
\begin{equation}
\label{eqn:iso_gaussian}
k_{\epsilon}(x,y) =
\exp\left(-\frac{||x-y||^2}{2\epsilon}\right).
\end{equation}
The user-chosen parameter $\epsilon >0$ is the \emph{kernel bandwidth parameter} (or the \emph{scaling  parameter}). 
The original diffusion map algorithm~\cite{coifman2006diffusion} requires an
isotropic kernel $h_{\epsilon}(||x - y||^{2})$ with exponential decay as $\|x-y\|\rightarrow\infty$. 
Let us describe the construction of a diffusion map following the steps in~\cite{coifman2006diffusion}.
 The kernel $k_{\epsilon}(x,y)$ is used to define  
an $n \times n$ similarity matrix $K_{\epsilon}$ with 
$[K_{\epsilon}]_{ij} = k_{\epsilon}(x_i, x_j)$.
The strong law of
 large numbers implies that for a scalar $f(x)$ on $\R^{d}$ we have:
\begin{equation}
\label{eqn:montecarlo} 
 \lim \limits_{n \to \infty}\frac{1}{n} \sum\limits_{j = 1}^{n}
 k_{\epsilon} (x_i, x_j) f(x_j) = \int_{\R^{d}} k_{\epsilon} (x_i, y) f(y)
 \rho(y) dy ~~\text{ almost surely.}
\end{equation}
It follows from the Central Limit Theorem that the error of this estimate decays as $\mathcal{O}(n^{-\frac{1}{2}})$. 
The action of the kernel
on sufficiently large datasets is approximated by an integral operator
$\mathcal{G}_{\epsilon}$ defined by 
\begin{equation}
\label{eqn:aniso_integral}
(\mathcal{G}_{\epsilon}f)(x) :=\int_{\R^{d}} k_{\epsilon}(x, y) f(y) dy.
\end{equation}
Namely, for a sufficiently large dataset, 
\begin{equation}
    \label{eqn:kernel_limit_density}
 \lim \limits_{n \to \infty}\frac{1}{n} \sum\limits_{j = 1}^{n}
 k_{\epsilon} (x_i, x_j) f(x_j) 
 = \mathcal{G}_{\epsilon}(f\rho)(x_i)
 \text{ almost surely}.
\end{equation}

The main innovation of the diffusion map algorithm \cite{coifman2006diffusion} in comparison with Laplacian eigenmap \cite{belkin2003laplacian}
is the introduction of the parameter $\alpha\in\R$ allowing us to control the influence of the density $\rho$ and approximate a whole family of differential operators.
Let us review the construction of diffusion maps.
The first step is to compute the normalizing factor $\rho_{\epsilon}(x)$, where
\begin{equation}
    \label{eqn:mmap_density}
\rho_{\epsilon}(x) = \int_{\R^{d}} k_{\epsilon}(x, y) \rho(y) dy.
\end{equation}
The \emph{right-normalized} kernel is
defined as
\begin{equation}
\label{eqn:rnkernel}
k_{\epsilon,\alpha}(x,y) = \frac{k_{\epsilon}(x,y)}{\rho_{\epsilon}^{\alpha}(y)}.
\end{equation}
We note that one can write the action of the right-normalized kernel on density-weighted functions $f(x)\rho(x)$ as
$$
\int_{\R^d} k_{\epsilon, \alpha}(x, y) f(y) \rho(y) dy = \int_{\R^d} k_{\epsilon}(x, y)
f(y) \rho(y)\rho_{\epsilon}^{-\alpha}(y) dy.
$$
Hence the right normalization regulates the influence of the density $\rho(x)$ in Monte Carlo integrals.
We define vector $p_{\epsilon}$ as the vector of row sums of the matrix $K_{\epsilon}$: 
\begin{equation}
\label{eqn:disc_pe}
    [p_{\epsilon}]_{i} = \sum\limits_{j = 1}^{n}
   [K_{\epsilon}]_{i,j}.
\end{equation}
We define the diagonal matrix  $D_{\epsilon} = {\sf diag}(p_{\epsilon})$.
Then, the discrete counterpart for the  \emph{right-normalized} kernel \eqref{eqn:rnkernel} is
\begin{equation}
\label{eqn:Kea}
K_{\epsilon,\alpha} :=
K_{\epsilon} D_{\epsilon}^{-\alpha}.
\end{equation}
Next, we fix 
\begin{equation}
\label{eqn:rhoea}
\rho_{\epsilon, \alpha}(x): = \int_{\R^{d}} k_{\epsilon,\alpha}(x, y) \rho(y) dy
\end{equation}
to \emph{left-normalize} the kernel, and define the Markov operator
$\mathcal{P}_{\epsilon, \alpha}$ on $f$ as
\begin{equation}
\label{eqn:Pea}
\mathcal{P}_{\epsilon, \alpha} f(x) = \frac{\int_{\R^{d}} k_{\epsilon,\alpha}(x,y) f(y) \rho (y) dy}{\rho_{\epsilon,\alpha}(x)}.
\end{equation}
Finally, we define a family of operators $\mathcal{L}_{\epsilon, \alpha}$ as
\begin{equation}
    \label{eqn:Lea}
\mathcal{L}_{\epsilon,\alpha}f = \frac{\mathcal{P}_{\epsilon,\alpha}f - f}{\epsilon}\equiv  \frac{1}{\epsilon}\left(\frac{\mathcal{G}_{\epsilon}(f \rho
\rho_{\epsilon}^{-\alpha})}{\mathcal{G}_{\epsilon}(\rho
\rho_{\epsilon}^{-\alpha})} - f\right).
\end{equation}
To obtain its discrete counterpart, we form a diagonal matrix
$D_{\epsilon,\alpha}$ with row sums of $K_{\epsilon,\alpha}$ along its diagonal and use it to \emph{left-normalize} the matrix $K_{\epsilon, \alpha}$ and 
get the Markov matrix
$P_{\epsilon, \alpha} = D^{-1}_{\epsilon, \alpha} K_{\epsilon, \alpha}.$
Then the family of \emph{discrete operators} $L_{\epsilon,\alpha}$ is defined as 
\begin{equation}
\label{eqn:discrete_gen}
L_{\epsilon,\alpha} := \frac{P_{\epsilon, \alpha} - I}{\epsilon}.
\end{equation}

The discrete  and continuous operators $L_{\epsilon,\alpha}$ and $\mathcal{L}_{\epsilon,\alpha}$ relate via the pointwise infinite data limit.
Let us fix an arbitrary point $x_1$ drawn from the invariant density $\rho$ and keep adding more points drawn from $\rho(x)$ into a dataset containing $x_1$. Then
\begin{equation}
\label{eqn:LL}
\lim_{n\rightarrow\infty}\left[L_{\epsilon,\alpha}[f]\right]_1 =\left[ \mathcal{L}_{\epsilon,\alpha}\right]f(x_1),\quad\text{almost surely}
\end{equation}
with error decaying as $O(n^{-\frac{1}{2}})$.
It is proven in \cite{coifman2006diffusion} that\footnotemark[1]
\begin{equation}
\label{eqn:Lelim}
\lim_{\epsilon\rightarrow 0} \mathcal{L}_{\epsilon,\alpha}f = \frac{1}{2}\left(\frac{\Delta(\rho^{1-\alpha}f) - f\Delta(\rho^{1-\alpha})}{\rho^{1-\alpha}}\right).
\end{equation}
\footnotetext[1]{We believe that a multiplicative constant is missing in Theorem 2 and in Proposition 10 in \cite{coifman2006diffusion}.}
If the input dataset $X$ comes from the overdamped Langevin SDE \eqref{eqn:overdamped_sde}, and hence
the invariant measure is Gibbs, i.e., $\rho(x) = Z^{-1}e^{-\beta V(x)}$, and the kernel is given by \eqref{eqn:iso_gaussian}, then
\begin{equation}
\label{eqn:dmap_generator_alpha}
\lim_{\epsilon\rightarrow 0} \mathcal{L}_{\epsilon,\alpha}f  = \frac{1}{2}\left[ \Delta f - 2\beta (1-\alpha) \nabla f^{\top} \nabla V\right].
\end{equation}
{ The series of steps described here leading to the construction of the matrix operator $L_{\epsilon,\alpha}$  \eqref{eqn:discrete_gen} will be referred to as {\tt dmap} (an abbreviation for the diffusion map algorithm).}

The renormalization parameter $\alpha$ tunes the
influence of the density $\rho$ in the operator $\mathcal{L}_{\epsilon,\alpha}$. Setting $\alpha =
0$ yields, up to a multiplicative constant $\sfrac{1}{2}$, the standard graph Laplacian, which converges to the Laplace-Beltrami operator
only if $\rho$ is uniform. With $\alpha = 1$ the density $\rho(x)$ is reweighted
so that the limiting density for $\epsilon \to 0, n \to \infty$ is uniform and
the Laplace-Beltrami operator $\mathcal{L}f =\tfrac{1}{2} \Delta f$ is recovered. Setting
$\alpha = \sfrac{1}{2}$ recovers the backward Kolmogorov operator, 
\begin{equation}
\label{eqn:bk}
\lim_{\epsilon\rightarrow 0} \mathcal{L}_{\epsilon,\sfrac{1}{2}}f = \frac{\beta}{2}\left[ \beta^{-1}\Delta f - \nabla f^{\top} \nabla V\right] \equiv \frac{\beta}{2}\mathcal{L}f,
\end{equation}
which is needed for computing the committor.

The diffusion map algorithm has also seen many other modifications and
improvements which are used heavily in practice.
A primary development concerns the kernel bandwidth parameter $\epsilon,$ which usually
requires extensive tuning in practice. Diffusion map variations such as
locally scaled diffusion maps~\cite{rohrdanz2011determination} and variable
bandwidth diffusion kernels~\cite{berry2016variable} utilize a bandwidth which
varies at each data point and can improve stability as well as accuracy at the boundary points of a dataset.

%
%
\subsection{Diffusion maps for data coming from SDEs with multiplicative noise}
\label{sec:dmmn}
An important limitation of the original class of diffusion maps with isotropic kernels \cite{coifman2006diffusion} 
is that it can only approximate infinitesimal generators of the form~\eqref{eqn:dmap_generator_alpha}
that are relevant only for gradient flows~\cite{banisch2020,berry2016local}.
This limitation is caused by the fact that the construction relies on the sampling density of the data but not
dynamical properties of the data. 
For example, the reversible diffusion process in collective variables~\eqref{eqn:cvsde} has the Gibbs distribution $\rho(x)$ and
diffusion matrix $M(x)$,
but application of diffusion maps will
approximate the generator of gradient dynamics with density $\rho(x)$ and
\emph{constant} diffusion matrix~\cite{banisch2020}. 
Hence, to approximate
generators for diffusion
processes with multiplicative noise, a modified diffusion maps algorithm is
required~\cite{singer2008,banisch2020}. 


\subsubsection{Mahalanobis diffusion maps}
\label{sec:mmap}
The foundational approach for applying diffusion maps to diffusion processes with multiplicative noise is proposed by Singer and Coifman (2008)~\cite{singer2008}.
They consider a diffusion process  
\begin{equation}
    \label{eqn:coifsing_orig_sde}
    dz_t = b(z_t)dt + \sqrt{2}dw_t,\quad z_t\in\mathcal{M},
\end{equation}
where $\mathcal{M}$ is a $d$-dimensional manifold. The generator for this process is given by 
\begin{equation}
\label{eqn:coifsing_orig_generator}
    \mathcal{L}_z = \Delta + b\cdot \nabla.
\end{equation}
The
dynamics~\eqref{eqn:coifsing_orig_sde} are considered as the unobserved intrinsic
dynamics of interest, while the observed dynamics is the process $x_t =
\zeta(z_t),$
where $\zeta$ is an injective, smooth function from $\mathcal{M}$ to $\R^{m},$
where $m \ge d.$ 
To elucidate the key idea from \cite{singer2008}, we assume that $\mathcal{M}\equiv\R^d,$ $m=d,$ and that the mapping $x = \zeta(z)$ is a diffeomorphism.
From Ito's Lemma, the differential of $\zeta$ is
\begin{equation}
\label{eqn:itos}
d\zeta_i(z_t) = \sum\limits_{j=1}^d\left(\frac{\partial \zeta_i
(z_t)}{\partial z_j} b_j(z_t) + \frac{\partial^2 \zeta_i(z_t)}{\partial
z_j^2}\right) dt + \sqrt{2}\sum\limits_{j=1}^{d} \frac{\partial
\zeta_i(z_t)}{\partial z_j}[dw_t]_j.
\end{equation}
Thus, the noise term for $d\zeta(z_t)$ is given by $\sqrt{2} J(z_t) dw_t,$ where $J(z_t)$ is the Jacobian of $\zeta$ with entries $[J(z)]_{ij} = \frac{\partial \zeta_i}{\partial z_j}.$

The crucial fact utilized in~\cite{singer2008} is 
the following relationship between the {$(JJ^{\top})^{-1}$}-weighted quadratic form in the space of observed variables $x = \zeta(z)$
and the Euclidean distance in the $z$-space:
\begin{align}
& 
\frac{1}{2}(\zeta(z) - \zeta(z'))^{\top}\left[(JJ^{^\top})^{-1}(x) + (JJ^{^\top})^{-1}(y)\right](\zeta(z) - \zeta(z')) \notag\\
&\quad \ = ||z-z'||^{2} + \mathcal{O}(||z - z'||^{4}). \label{eqn:mahal_jacob}
\end{align}
This relationship motivated the introduction of the anisotropic kernel
\begin{equation}
    \label{eqn:singerkernel}
    k_{\epsilon}(x, y) =
\exp\Big(-\frac{1}{4\epsilon}(x - y)^{\top}((JJ^{\top})^{-1}(x) + (JJ^{\top})^{-1}(y))(x - y) \Big).
\end{equation}
The diffusion map with kernel \eqref{eqn:mkernel} approximates the generator $\mathcal{L}_z$ for the unknown latent dynamics $z_t$ in the case where $b(z_t) = -\nabla U(z_t)$ for some potential $U(z)$.
{The fact that the diffusion matrix is of the form  $JJ^{\top}$ is essential for the proof of this approximation presented in \cite{singer2008}. }
Relationship \eqref{eqn:mahal_jacob} essentially reduces this to the proof for diffusion maps with rotationally symmetric Gaussian kernel in \cite{coifman2006diffusion}.

The { algorithm proposed in \cite{singer2008}} and variants have been applied to multiscale fast slow-processes~\cite{singer2008, dsilva2016data}, nonlinear filtering problems~\cite{talmon2013empirical}, optimal transport and data fusion problems~\cite{moosmuller2020geometric, dietrich2020manifold}, chemical reaction
networks~\cite{singer2009detecting}, localization in sensor
networks~\cite{schwartz2019}, and molecular dynamics~\cite{dsilva2013nonlinear}.
Further,  kernel \eqref{eqn:mkernel} has recently been {used for isometric embeddings to high-dimensional latent spaces~\cite{peter2020}
and for deep learning
frameworks~\cite{schwartz2019, peter2020}}.

{ This work addresses the case where diffusion matrix $M(x)$ in SDE \eqref{eqn:cvsde} \emph{is not necessarily decomposable as $M(x)=\left(JJ^{\top}\right)(x)$} and hence there is no relationship of the form \eqref{eqn:mahal_jacob} to utilize. Following \cite{singer2008}, we use the Mahalanobis kernel
\begin{equation}
    \label{eqn:mkernel}
    k_{\epsilon}(x, y) =
\exp\Big(-\frac{1}{4\epsilon}(x - y)^{\top}(M^{-1}(x) + M^{-1}(y))(x - y) \Big).
\end{equation} 
Other than the choice of the kernel, the Mahalanobis diffusion map algorithm ({\tt mmap}) follows the steps of the diffusion map algorithm ({\tt dmap}) detailed in Section \ref{sec:diffusion_maps}. For the reader's convenience, we summarize {\tt mmap} in Algorithm \ref{alg:mmap}. The family of differential operators approximated by {\tt mmap} will be derived in Section \ref{sec:main}.
}

{
\begin{algorithm}
  \SetAlgoNoLine
  \DontPrintSemicolon
  \KwIn{data $X = \{x_i\}_{i=1}^n$, diffusion matrices $\{M(x_i)\}_{i=1}^n$, bandwidth $\epsilon$, renormalization parameter $\alpha$ }
  \KwOut{Matrix operator $L_{\epsilon,\alpha}$}
    \nonl \custom{\textnormal{Construct kernel using~\eqref{eqn:mkernel}}}{
  $\displaystyle [K_{\epsilon}]_{i,j}  = k_{\epsilon}(x_i, x_j), ~~i,j=1,\ldots,n $ \;}
    \nonl \custom{\textnormal{Find row sums of the kernel matrix and form a diagonal matrix}}{
   $\displaystyle [p_{\epsilon}]_i = \sum_{j=1}^n [K_{\epsilon}]_{ij},~~ i=1,\ldots, n$\;
   $\displaystyle D_{\epsilon} = {\sf diag}\{[p_{\epsilon}]_1,\ldots,[p_{\epsilon}]_n \}$\;      
  }
  \nonl\custom{\textnormal{Right normalize the kernel}}{
  $\displaystyle[K_{\epsilon,\alpha}] :=K_{\epsilon}D_{\epsilon}^{-\alpha}$ }
   \nonl\custom{\textnormal{Left normalize the kernel}}{
     $\displaystyle [p_{\epsilon,\alpha}]_i = \sum_{j=1}^n [K_{\epsilon,\alpha}]_{ij}, ~~i=1,\ldots, n$\;
   $\displaystyle D_{\epsilon,\alpha} = {\sf diag}\{[p_{\epsilon,\alpha}]_1,\ldots,[p_{\epsilon,\alpha}]_n \}$\;      
  $\displaystyle P_{\epsilon, \alpha}:=  D_{\epsilon,\alpha}^{-1}K_{\epsilon,\alpha}$}
  \nonl\custom{\textnormal{Construct generator}}
  {$\displaystyle L_{\epsilon,\alpha} = \frac{P_{\epsilon,\alpha} - I}{\epsilon}$ \;}
  \caption{Mahalanobis Diffusion Map Algorithm ({\tt mmap}) }\label{alg:mmap}
\end{algorithm}
}

\begin{remark}
The term \emph{Mahalanobis kernel} is related to the Mahalanobis distance. If data points $x$ and $y$ are sampled from a multivariate Gaussian distribution with covariance matrix {$C$, then 
$$
(x-y)^\top C^{-1} (x-y)
$$
is the squared Mahalanobis distance between $x$ and $y$. In the orthonormal basis of eigenvectors of $C$},
the difference between each component of $x$ and $y$ is normalized by the corresponding variance, which reflects the difficulty to deviate along each direction. 
Hence,  kernel \eqref{eqn:mkernel} is a decaying exponential function of an approximate squared Mahalanobis distance. 
Therefore, it is designed to account for anisotropy of the diffusion process where the data is coming from.
\end{remark}

\subsubsection{Local kernels}
\label{sec:lk}
A further development facilitating data-driven analysis of anisotropic diffusion processes was  done by Berry and Sauer (2016) \cite{berry2016local}. 
Their local kernels theory generalizes theoretical results of~\cite{coifman2006diffusion,singer2008} to 
a class of anisotropic kernels which utilize user-defined drift
 vectors $b(x_i)$ and diffusion matrices $A(x_i)$. 
 The local kernel approach has been extended to related work in solving elliptic PDEs with diffusion maps~\cite{gilani2019} 
 and to computing reaction coordinates for molecular simulation~\cite{banisch2020}.
 In~\cite{banisch2020} the authors incorporate arbitrary sampling densities into the local kernel 
 approach~\cite{berry2016local} and prove that {for user-defined drift vector $b(x)$ and user-defined diffusion matrix $A(x)$}, kernels of the form 
\begin{equation} 
    \label{eqn:local_kernel}
 k_{\epsilon}^{A,b}(x,y) = \exp\left(-\frac{1}{4 \epsilon}(x - y + \epsilon b(x))^{\top} A^{-1}(x)(x - y + \epsilon b(x)) \right),
\end{equation}
applied to data $\{x_i\}_{i=1}^{n}$ with arbitrary density can be normalized (similarly to the diffusion map with $\alpha = 1$) to approximate the differential operator
\begin{equation}
 \label{eqn:generator}
 \mathcal{L}f(x) = b(x) \nabla f(x) + {\sf tr}{[A(x) \nabla\nabla f(x)]}.
\end{equation}
Equation~\eqref{eqn:generator} describes a broader class of
 generators than \eqref{eqn:mgen}, which is advantageous. 
 On the downside, the local kernel diffusion map
 algorithm of~\cite{banisch2020} requires drift estimates at all data points, as well as a second kernel 
 $k_{\tilde{\epsilon}}(x,y)$ with additional scaling parameter $\tilde{\epsilon}$
in order to normalize the density of the dataset. 
As a result, implementation of the local kernel approach requires adjustment of two scaling parameters $\epsilon$ and $\tilde{\epsilon}$, which can be challenging.

We are primarily interested in the reversible dynamics in collective variables coming from MD simulations.
On the other hand, the density of the data is far from being uniform, and may change by orders of magnitude thereby complicating the tuning of scaling parameters.
Therefore, we choose to use {\tt mmap} rather than the local kernel approach. 

%
%

\section{Theoretical results}
\label{sec:main}
{
Our goal is to prove that the {\tt mmap} algorithm (Algorithm \ref{alg:mmap}) with $\alpha = \sfrac{1}{2}$ approximates the generator \eqref{eqn:mgen} for SDE \eqref{eqn:cvsde}, the overdamped Langevin dynamics in collective variables. First we show that not every symmetric positive definite smooth matrix function $M(x)$ admits the decomposition $JJ^\top(x)$ where $J(x)$ is the Jacobian matrix function for some smooth vector-function. This will justify the lengthy calculation conducted in our proof of the main theorem (Theorem \ref{thm:main-theorem} below). Next, we derive the family of differential operators approximated by {\tt mmap} with an arbitrary $\alpha\in\mathbb{R}$ (Theorem \ref{thm:main-theorem}). Finally, we evaluate the resulting differential operator at $\alpha = \sfrac{1}{2}$ and show that it is the generator for \eqref{eqn:cvsde} (Corollary \ref{thm:main_corollary}).
}

{
\subsection{Not every diffusion matrix is associated with a variable change}
\label{sec:MnotJJ}
The fact that not every smooth symmetric positive definite matrix function $M(x)$ can be decomposed as
\begin{equation}
\label{eqn:Mdec}
M(x) = J(x)J(x)^\top\quad{\rm where}\quad J = \left(\frac{\partial f_i}{\partial x_j}\right)_{i,j = 1}^d
\end{equation}
for some smooth vector-function $f:\Omega\rightarrow\R^d$, where $\Omega$ is an open set in $\mathbb{R}^d$, is not new but it is not widely known. The non-existence of decomposition \eqref{eqn:Mdec} is pointed out for the position-dependent diffusion matrix in the Moro-Cardin 2D example \cite{MoroCardin1998} by  M. Johnson and G. Hummer \cite{JohnsonHummer2012} with a reference to the textbook by H. Risken \cite{risken1996fokker} (Section 4.10), where the general criterion for the existence of decomposition \eqref{eqn:Mdec} consisting in vanishing a certain complicated differential form is presented. 

Here we will give a simple proof of the fact that not every symmetric positive definite smooth matrix function admits decomposition \eqref{eqn:Mdec} by establishing a necessary condition for a class of $2\times 2$ symmetric positive definite matrix functions
\begin{equation}
\label{eqn:Mclass}
 M(x,y) = m^2(x,y)I_{2\times 2},\quad m(x,y)\in C^2(\Omega),\quad m(x,y) > 0~\forall(x,y)\in\Omega,
\end{equation}
for decomposition \eqref{eqn:Mdec} to exist. Precisely, the necessary condition requires the function $\log m(x,y)$ to be \emph{harmonic}. Moreover, if the open set $\Omega$ is simply connected, this condition is also sufficient for the class \eqref{eqn:Mclass}. 
\begin{theorem}
\label{thm:lem6}
Let $M(x,y)$ be a symmetric positive definite matrix function of the form \eqref{eqn:Mclass} where $m(x,y)$ is a positive twice continuously differentiable function in an open set $\Omega\subset \mathbb{R}^2$ and $I_{2\times 2}$ is a $2\times 2 $ identity matrix. Suppose that $M$ admits decomposition $M(x,y) = JJ^\top(x,y)$ where $J(x,y)$ is the Jacobian matrix of some twice continuously differentiable vector-function $f:\Omega\rightarrow\mathbb{R}^2$. Then $\log m(x,y)$ must be harmonic, i.e.,
\begin{equation}
\label{eqn:harmonic}
\left(\frac{\partial^2}{\partial x^2} +\frac{\partial^2}{\partial y^2} \right)\log m(x,y) = 0\quad\forall(x,y)\in\Omega.
\end{equation}
Moreover, if the open set $\Omega$ is simply connected, then  \eqref{eqn:harmonic} is also a sufficient condition for the existence of decomposition \eqref{eqn:Mdec}. If $\Omega$ is not simply connected, \eqref{eqn:harmonic} is not a sufficient condition. 
\end{theorem}
A proof of Theorem \ref{thm:lem6} is given in \ref{appendix:not_every_diffusion}.

Thus, any matrix function of the form \eqref{eqn:Mclass} where $\log m(x,y)$ is not harmonic in $\Omega$ does not admit  decomposition \eqref{eqn:Mdec} in $\Omega$. For example, the function $m$ in the Moro-Cardin example \cite{MoroCardin1998} 
\begin{equation}
\label{eqn:MoroCardin}
m(x,y) = \left(1+e^{-\frac{x^2+y^2}{2}}\right)^{-1/2}
\end{equation}
is such that its logarithm is not harmonic:
$$
\Delta\log m(r) = \frac{1}{2}\left[\frac{(2-r^2)(1+e^{-r^2/2})e^{-r^2/2}+r^2e^{-r^2}}{(1+e^{-r^2/2})^2}\right],\quad r=\sqrt{x^2+y^2}.
$$

Furthermore, any $d\times d$ twice continuously differentiable matrix function $M(x,y)$ that has a principal $2\times 2$ submatrix of the form \eqref{eqn:Mclass} where $\log m(x,y)$ is not harmonic in $\Omega$ does not admit decomposition \eqref{eqn:Mdec}.
}

 \subsection{The family of differential operators approximated by {\tt mmap}}
We will adopt three technical assumptions.
The first one deals with the space of collective variables $x$:
{
\begin{assumption}
\label{Ass1}
The range of $x$ representing the set of collective variables constitutes a $d$-dimensional manifold $\mathcal{M}$ which is either $\mathbb{R}^d$, or the $d$-dimensional torus $\mathbb{T}^d$, or a direct product of torus $\mathbb{T}^{k}$ and $\mathbb{R}^{d-k}$ . In all cases, $\mathcal{M}$ is of the form
\begin{equation}
\label{eqn:manifold}
\mathcal{M} =\mathbb{T}^{k}\times\mathbb{R}^{d-k} ,
\quad\text{for some}\quad 0\le k\le d.
\end{equation}
By the torus $\mathbb{T}^k$, $1\le k\le d$, we mean the ``flat" torus, i.e., the direct product of intervals with periodic boundary conditions, i.e.,
\begin{align*}
&\mathbb{T}^k\times \mathbb{R}^{d-k} = [a_1,b_1]\times[a_2,b_2]\times\ldots\times[a_k,b_k]\times \mathbb{R}^{d-k},\\
&(x_1,\ldots,x_{l-1},a_l,x_{l+1},\ldots,x_d) = (x_1,\ldots,x_{l-1},b_l,x_{l+1},\ldots,x_d),
~~1\le l\le k.
\end{align*}
The metric on such a torus is locally Euclidean \cite{NikShaf}, i.e., within any open ball of radius 
\begin{equation}
\label{eqn:Reuc}
R_{Euc}: = \min_{1\le l\le k}\frac{|b_l-a_l|}{2}.
\end{equation} 
Therefore, the metric on $\mathcal{M}$ is Euclidean if $\mathcal{M}=\mathbb{R}^d$ or locally Euclidean within any ball of radius $R_{Euc}$ if $\mathcal{M} =  \mathbb{T}^{k}\times\mathbb{R}^{d-k}$ for some $1\le k\le d$.
\end{assumption}
}

Assumption \ref{Ass1} is  nonrestrictive in view of chemical physics applications, as usually collective variables are dihedral angles or distances between certain atoms.
For example, the alanine dipeptide molecule is represented in two or four dihedral angles, $\mathcal{M} = \mathbb{T}_2$ or $ \mathbb{T}_4$, a 2D or a 4D torus respectively. 
Assumption \ref{Ass1} allows us to prove our main theoretical result from scratch using only elementary tools.

The second and third assumptions impose integrability and differentiability conditions on the diffusion matrix $M(x)$ and a class of functions $f:\mathcal{M}\rightarrow\R$ 
to which we apply the constructed family of operators. 
We need the following definition:
\begin{definition}
We say that a continuous function $f:\R^d\rightarrow \R$ grows not faster than a polynomial as $\|x\|\rightarrow\infty$ if
there exist constants $A\ge 0$, $B\ge 0,$ and $l\in\mathbb{N}$ such that 
$$
|f(x)|\le A + B\|x\|^l \quad\forall x\in\R^d.
$$
\end{definition}
\begin{assumption}
\label{Ass2}
The diffusion matrix $M(x)$ is symmetric positive definite. Its inverse $M^{-1}(x)$ is a four-times continuously differentiable matrix-valued function $M^{-1}:\mathcal{M}\rightarrow\R^{d\times d}$ and the determinant of $M^{-1}(x)$ is bounded away from zero. 
If the manifold $\mathcal{M}$ is unbounded {(i.e., $0\le k\le d-1$ in \eqref{eqn:manifold})}, then the entries $(M^{-1})_{ij}(x)$ and their first derivatives $\frac{\partial (M^{-1})_{ij}(x)}{\partial x_{\ell}}$ grow not faster than a polynomial as $\|x\|\rightarrow\infty$.
\end{assumption}
\begin{assumption}
\label{Ass3}
The function $f(x)$ is four-times continuously differentiable. If $\mathcal{M}$ is unbounded then $f(x)$ grows not faster than a polynomial as $\|x\|\rightarrow\infty$.
\end{assumption}

Now we are ready to formulate our convergence results for {\tt mmap}.
\begin{theorem}
\label{thm:main-theorem}
Suppose a manifold $\mathcal{M}$ and a diffusion matrix $M(x):\mathcal{M}\rightarrow\R^{d\times d}$ satisfy Assumptions \ref{Ass1} and  \ref{Ass2} respectively.
Let $\alpha\in\R$ be fixed and the kernel $k_{\epsilon,\alpha}$ be the Mahalanobis kernel \eqref{eqn:mkernel}, and 
the operator $\mathcal{L}_{\epsilon,\alpha}$ be constructed according to  \eqref{eqn:mmap_density},
\eqref{eqn:rnkernel},
\eqref{eqn:rhoea},
\eqref{eqn:Pea}, and 
\eqref{eqn:Lea}.
Then for any function $f(x):\mathcal{M}\rightarrow \R$ satisfying Assumption \eqref{Ass3} we have
\begin{align}
    \label{eqn:limiting_generator}
\lim_{\epsilon\rightarrow 0} \mathcal{L}_{\epsilon, \alpha}f(x) &= 
\frac{1}{2}\left(
\frac{{\sf tr}
\left(M 
\left[ \nabla\nabla \left[\rho^{1-\alpha}f\right] - f\nabla\nabla \rho^{1-\alpha}\right]   
\right)
}{\rho^{1-\alpha}} 
\right) 
\nonumber\\
&+
\frac{\alpha}{2}\left(
\frac{{\sf tr}
\left(M
\left[ 
\nabla\left[\rho^{1-\alpha}f\right]  - 
f\nabla\left[\rho^{1-\alpha}\right] 
\right]   \frac{\nabla|M|^{-\top}}{|M|^{-1}}
\right)
}{\rho^{1-\alpha}} 
 \right) \nonumber \\
 &-\left(\frac{ [\nabla(f\rho^{1-\alpha})-f\nabla(\rho^{1-\alpha})]^\top \omega_1]}{\rho^{1-\alpha}}
\right)\quad \forall x\in\mathcal{M},
\end{align}
where $|M|$ denotes the determinant of $M$, and $\omega_1(x)$ is a vector-valued function defined by
\begin{equation}
\label{eqn:omega1def}
\omega_{1,i}(x):=
 \frac{|M(x)|^{-1/2}}{(2\pi\epsilon)^{d/2}} \frac{1}{4\epsilon^2}\int_{\mathcal{M}}e^{-\frac{z^\top M(x)^{-1} z}{2\epsilon} }z_i \left[z^\top [\nabla M^{-1}(x){z}]z\right]dz.
\end{equation}
\end{theorem}
A proof of this theorem is done by a direct calculation of limit \eqref{eqn:limiting_generator}. 
It is carried out from scratch and involves only elementary tools from linear algebra and multivariable calculus. 
It is found in~\ref{appendix:main_proof}.
We remark that, in turn, the corresponding discrete operator applied to $f(x)$ discretized to a point cloud drawn from the invariant density $\rho(x)$ and with $L_{\epsilon,\alpha}[f]$ defined by
\eqref{eqn:disc_pe},
\eqref{eqn:Kea},
\eqref{eqn:discrete_gen}
 converges pointwise with probability one to $\mathcal{L}_{\epsilon,\alpha}f$ as the number of data points tends to infinity.

Equation \eqref{eqn:limiting_generator} defines a family of differential operators parametrized by $\alpha\in\R$. 
Since our goal is to compute committors, we are primarily concerned with approximating the generator \eqref{eqn:mgen} 
for the overdamped Langevin SDE in collective variables \eqref{eqn:cvsde}.
Setting $\alpha=\sfrac{1}{2}$ approximates the generator that we need:
\begin{corollary}
    \label{thm:main_corollary}
Let $\mathcal{M}$ and $M(x)$ be as in Theorem \ref{thm:main-theorem}.  
Suppose that the invariant density $\rho(x)$ takes the form of the Gibbs distribution $\rho(x) = Z^{-1} e^{-\beta F(x)}$ 
for free energy $F,$ temperature parameter $\beta^{-1},$ and normalizing constant $Z = \int_{\mathcal{M}} e^{- \beta F(x)} dx$. 
Then for $\alpha=\sfrac{1}{2}$ the limit \eqref{eqn:limiting_generator} reduces to
\begin{equation}
\label{eqn:main-eq}
\lim_{\epsilon\rightarrow 0} \mathcal{L}_{\epsilon,\sfrac{1}{2}} f(x) = \frac{\beta}{2}\mathcal{L}f(x)\quad\forall x\in\mathcal{M},
\end{equation}
where 
\begin{equation}
\label{eqn:Lgen}
\mathcal{L}f = \left(- M\nabla F + \beta^{-1}\left(\nabla\cdot M\right)\right)^\top\nabla f + \beta^{-1}{\sf tr}[M\nabla\nabla f]
\end{equation}
is the generator for the SDE 
\begin{equation}
\label{eqn:main-sde}
dx_t =\left[ -M(x_t)\nabla F(x_t) + \beta^{-1}\nabla\cdot M(x_t)\right]dt +\sqrt{2\beta^{-1}}M^{1/2}(x_t)dw_t.
\end{equation}
\end{corollary}
A proof of Corollary \ref{thm:main_corollary} is found in~\ref{appendix:alpha0.5}. { Our main interest is in solving the committor PDE \eqref{eqn:comm_pde}. Our approach consists in approximating the generator $\mathcal{L}$ in  \eqref{eqn:comm_pde} by the matrix operator $L_{\epsilon,1/2}$ which converges to $\mathcal{L}_{\epsilon,1/2}$ as the number of data points $n$ tends to infinity as $O(n^{-1/2})$.}

Finally, we remark that the use of the symmetric Mahalanobis kernel \eqref{eqn:mkernel} is essential for the convergence of 
$\mathcal{L}_{\epsilon,\sfrac{1}{2}} f(x) $ to the generator \eqref{eqn:Lgen}. 
If one processes data sampled from a long trajectory of SDE \eqref{eqn:main-sde} with {\tt dmap}, i.e., 
implements the isotropic Gaussian kernel \eqref{eqn:iso_gaussian} in the diffusion map algorithm,
one obtains an approximation to the generator for the diffusion process governed by 
$$
dx_t =-\nabla F(x_t)  +\sqrt{2\beta^{-1}}dw_t
$$
which has the same invariant density $Z^{-1}\exp(-\beta F(x))$ as  \eqref{eqn:main-sde}  but a different drift and a different diffusion matrix. 
If one replaces the half-sum $\tfrac{1}{2}(M(x)+M(y))$ in  \eqref{eqn:mkernel} with $M(x)$, all terms containing derivatives of 
$M$ in \eqref{eqn:limiting_generator} do not arise, and only the first term in \eqref{eqn:limiting_generator} remains.
For $\alpha=\sfrac{1}{2}$ this yields $-M\nabla F\cdot\nabla f + \beta^{-1}{\sf tr}[M\nabla\nabla f]$, the generator for the dynamics 
$$
dx_t = -M(x_t)\nabla F(x_t)dt +\sqrt{2\beta^{-1}}M^{1/2}(x_t)dw_t,
$$
which approximates \eqref{eqn:main-sde} only if $M$ is constant or $\beta^{-1}$ is small.
In our examples presented in the next section, $M$ varies considerably and $\beta^{-1}$ is not so small, rendering the term $\beta^{-1}\nabla\cdot M(x_t)$ non-negligible.

%
%
\section{Examples}
\label{sec:examples}
In this section, we test {\tt mmap} on two examples: alanine dipeptide and Lennard-Jones-7 in 2D.
The results obtained with {\tt mmap} will be validated by comparing them to results of other established methods and contrasted to those 
of the diffusion map with isotropic Gaussian kernel ({\tt dmap}).
In view of the remark at the end of Section \ref{sec:main}, the fact that the committors obtained using {\tt dmap} 
are significantly less accurate than the {\tt mmap} committors is unsurprising.

%
%

\subsection{Transitions between metastable states C5 and C7eq in alanine dipeptide}
\label{sec:aladip}

\begin{figure}[htbp]
\begin{center}
(a)\includegraphics[width = 0.45\textwidth]{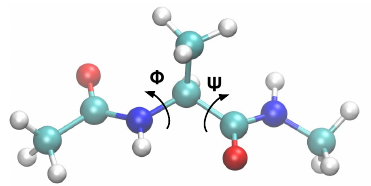}
(b)\includegraphics[width = 0.45\textwidth]{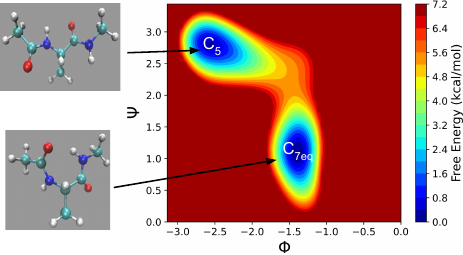}
\caption{        (a):
        Structure of alanine dipeptide and dihedral  angles $\Phi$ and  $\Psi$ serving as collective variables. 
        (b): Free energy surface of alanine dipeptide in vacuum at temperature $T=300$K in vicinity of C5 and C7eq minima, in $\Phi, \Psi$ coordinates.
}
\label{fig:aladip}
\end{center}
\end{figure}
Alanine dipeptide, a small biomolecule comprising  22 atoms, is  a popular test
example in chemical
physics~\cite{2006string,cameron2013estimation,plumed2,trstanova2020local}.
A typical set of collective variables effectively representing its motion consists of four or just two dihedral angles.
We choose the set of only two dihedral angles $\Phi$ and $\Psi$ shown in Figure~\ref{fig:aladip}(a). Their range comprises a two-dimensional torus, i.e the manifold $\mathcal{M}$ is $\mathbb{T}^{2}.$

\subsubsection{Obtaining input data}
\label{sec:ADdata}
\begin{figure}[htbp]
\begin{center}
 \includegraphics[width = 0.6\textwidth]{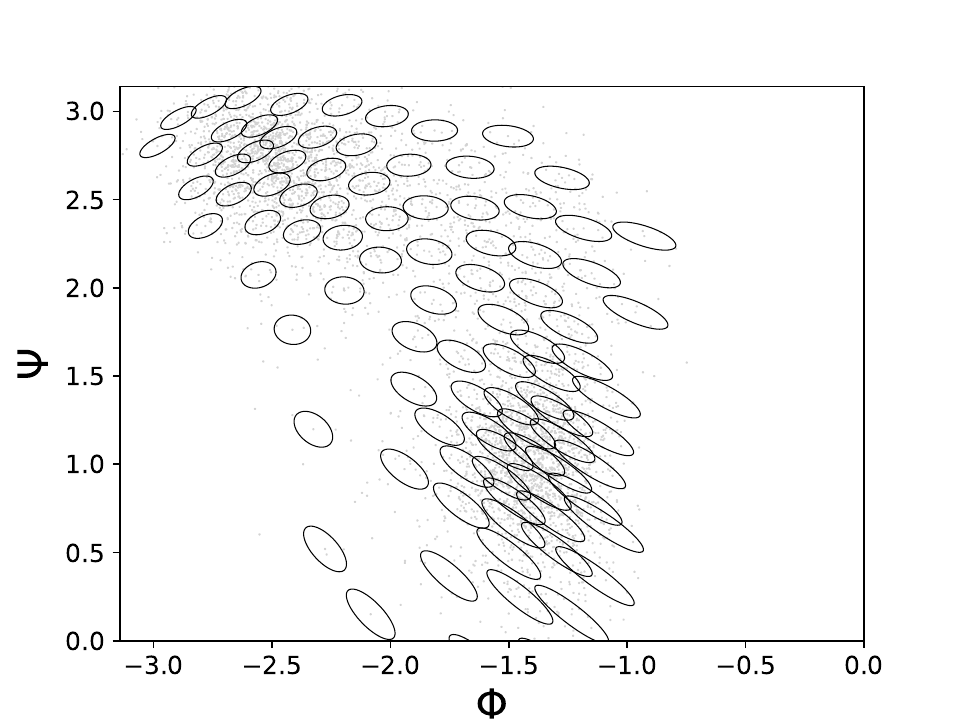}
\caption{Ellipses corresponding to the principal components of the estimated diffusion
        matrices for the alanine dipeptide data (faint grey dots). Each ellipse is plotted with
        center on the point whose diffusion matrix it represents. The ellipses are
        plotted on a representative subsampling of the trajectory data.}
\label{fig:c5c7_diffusion_tensors}
\end{center}
\end{figure}

We used a velocity-rescaling thermostat to set the temperature to 300K in a vacuum and ran a 1 nanosecond trajectory under constant number, volume, and temperature (NVT) conditions, integrating Newton's equations of motion with timestep 
$2$ femtoseconds using the molecular dynamics software GROMACS~\cite{van2005}.
For use with diffusion maps, we subsampled the trajectory at equispaced intervals of timesteps to obtain $n=5000$ data points $\{x_i\}_{i = 1}^{5000}$ with $x_i \in \mathbb{T}^{2}$. 
The diffusion matrices $M(x_i)$ were obtained { following the methodology of~\cite{2006string} (see~\ref{sec:diffusion_estimation})}
and are visualized in
Figure~\ref{fig:c5c7_diffusion_tensors}. 
The reactant and product sets $A$ and $B$ are the small ellipses centered at the 
the C5 and C7eq minima in the $(\Phi, \Psi)$-space shown in Figure \ref{fig:aladip}(b).
In $\Phi, \Psi$ coordinates the C5 and C7eq minima are $(-2.548, 2.744)$ and $(-1.419, 1.056)$ respectively, and the ellipses shown are the level sets of the free energy $F$ at $1.4$ kcal/mol.

To compare the committors computed via {\tt mmap} and {\tt dmap} 
with the one obtained by a traditional PDE solver, we discretized 
the range $[-\pi,\pi]^2$ of $(\Phi,\Psi)$ into a uniform square mesh $128\times 128$ as in \cite{cameron2013estimation}
and generated $M(x)$ and $\nabla F(x)$ using the procedure from~\cite{2006string}{, summarized in~\ref{sec:diffusion_estimation}.} We then posed a boundary-value problem for the committor PDE from \eqref{eqn:mgen}, 
\eqref{eqn:comm_pde} and solved it using a finite difference scheme with central second-order accurate approximations to the derivatives.

Often (see e.g. \cite{2006string,ren2019} and many other works) a much rarer transition in alanine dipeptide is studied: 
the one between the combined metastable state comprising C5 and C7eq 
and the metastable state called C7ax located near $\Phi=75^\circ$, $\Psi = -75^\circ$.
We chose the transition between  C5 and C7eq  for our tests because  it can be easily sampled at room temperature $T=300$K.
It is essential for {\tt mmap} to have sufficient 
data coverage of the transition region, and the data must be sampled from the invariant distribution.
The study of this transition gives us another benefit: unlike that for the transition between  
(C5,C7eq) and C7ax, the
free energy barrier between C5 and C7eq is not large in comparison with $k_bT$.
This renders the term $\beta^{-1}\nabla\cdot M(x)$ in SDE \eqref{eqn:cvsde} non-negligible which is nonzero if and only if $M(x)$ is nonconstant.
As a result, the contrast between the results of {\tt mmap} and {\tt dmap} is amplified.
We leave the task of upgrading {\tt mmap} to make it applicable to datasets obtained using enhanced sampling techniques for the future.


\subsubsection{Results and Validation}
\label{sec:c5c7_results}

We computed the committor using the {\tt mmap} and {\tt dmap} algorithms with a large range of values of the scaling parameter $\epsilon$. This range is naturally bounded from above  and below 
by the diameter of the point cloud and by the minimal distance between data points, respectively.
In addition, we computed the committor by solving the boundary-value problem for the committor PDE  using finite differences as mentioned in Section \ref{sec:ADdata} and took it as a ground truth $q_{true}$.
To quantify the error of the {\tt mmap} and {\tt dmap} committors, we evaluated the root-mean-square (RMS) error
$$
\text{RMS error} = \sqrt{\frac{\sum_{j=1}^n (q_{true}(x_j) - q_{approx}(x_j))^2}{n}},
$$
where $\{x_j\}_{j=1}^{n}$ are the data points. 
The finite difference solution $q_{true}$ was evaluated on the data through bilinear interpolation.
It is clear that the committor computed with diffusion maps cannot be expected to be accurate on the outskirts of the dataset where the data coverage is insufficient.
On the other hand, the {\tt mmap} and {\tt dmap} committors are exact at $A$ and $B$ by construction and highly accurate near them, 
and these are the regions containing the majority of data points as they are sampled from the invariant density.
We care the most about the accuracy of the {\tt mmap} and {\tt dmap} committors in the transition region. 
Therefore, we select the subset of points marked with magenta dots in Figure~\ref{fig:ADerr}(a).
The graphs of the RMS errors for  {\tt mmap} (red) and {\tt dmap} (blue) over this subset as functions of $\epsilon$ are displayed in Figure~\ref{fig:ADerr}(b). The epsilon values minimizing the RMS error of the {\tt mmap} and {\tt dmap} committor are $\epsilon=0.01$ and $\epsilon=0.003$ with RMS errors $0.014$ and $0.036$ respectively.
We know that the range of $\epsilon$ used for {\tt mmap} is shifted with respect to the range of {\tt dmap} due to the fact that the eigenvalues of $M^{-1}$ range from $1.16$ to $8.11$ and average to $4.06$. As a result, the optimal $\epsilon$ for {\tt mmap} is approximately larger than that for {\tt dmap} by a factor of $3.33$. Moreover, for all $\epsilon$-values in the overlap of ranges the error for {\tt mmap} is smaller than that of {\tt dmap}.
\begin{figure}[htbp]
\begin{center}
(a)\includegraphics[width=0.45\textwidth]{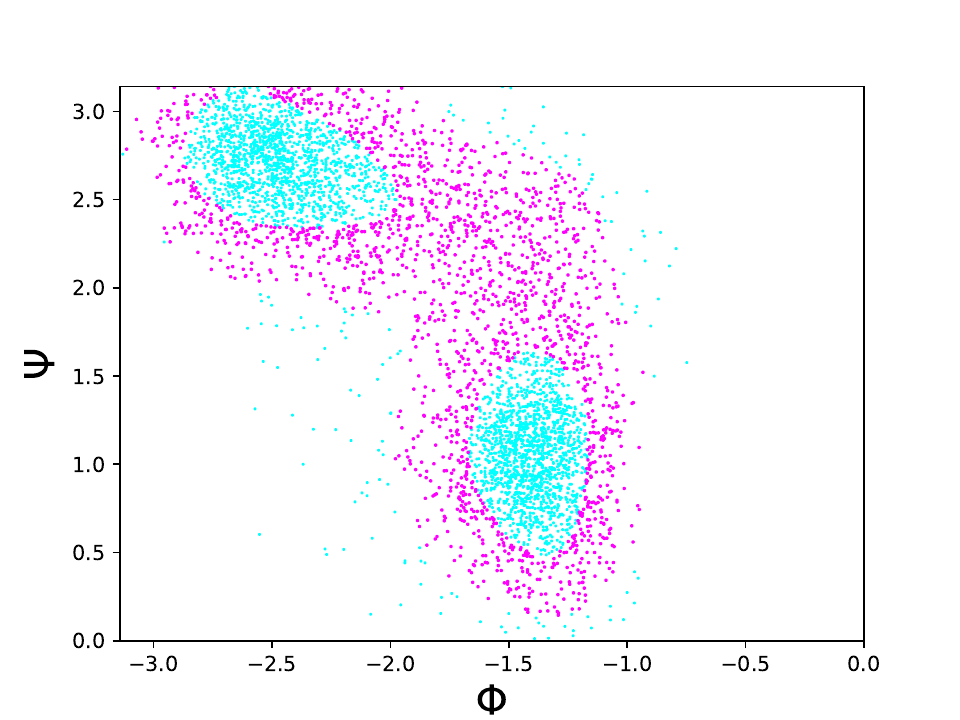}
(b)\includegraphics[width=0.45\textwidth]{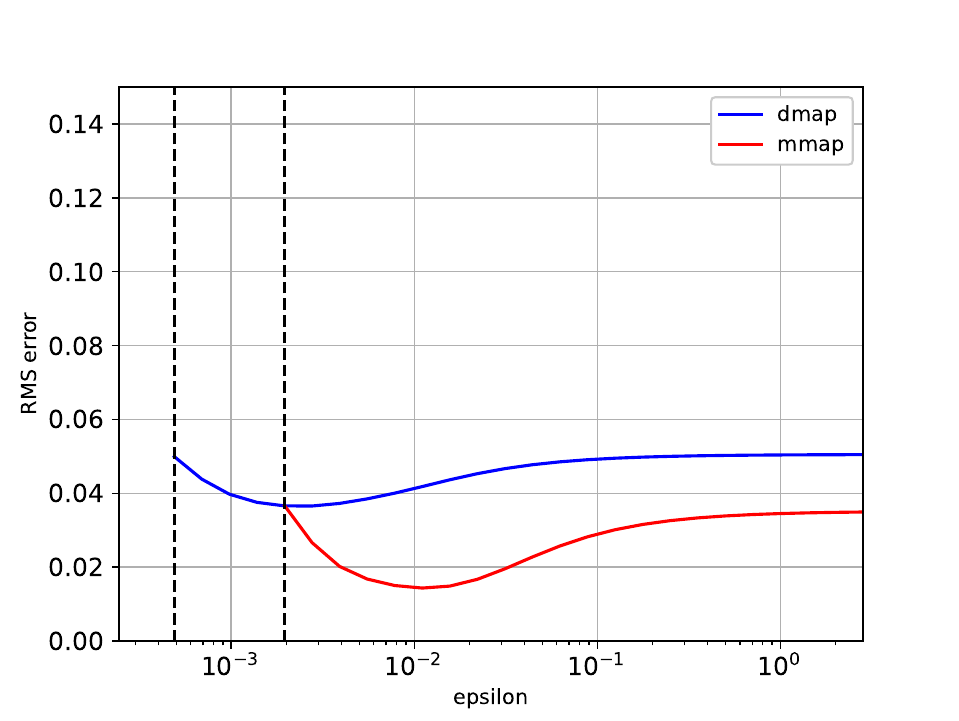}
\caption{(a): Alanine dipeptide dataset. Its subset lying in the region of interest marked with magenta is used for computing the RMS error.
(b): RMS errors for {\tt dmap} and {\tt mmap} committors as functions of the scaling parameter $\epsilon$.  The dotted lines indicate
        lower bound for $\epsilon$-values related to the minimal distance between data points.
}
\label{fig:ADerr}
\end{center}
\end{figure}
The level sets of the computed committor using {\tt mmap} and {\tt dmap} for the values of $\epsilon$ minimizing the error are shown, respectively, in Figure~\ref{fig:c5c7_committor}(a) and~\ref{fig:c5c7_committor}(b) with dashed lines. 
The solid lines are the corresponding level sets of the committor $q_{true}$ computed by finite differences.
The  level sets  of the {\tt mmap} committor closely match those of $q_{true}$, while the level sets of the {\tt dmap} committor notably deviate from them.
\begin{figure}[htbp]
\begin{center}
(a)\includegraphics[width=0.45\textwidth]{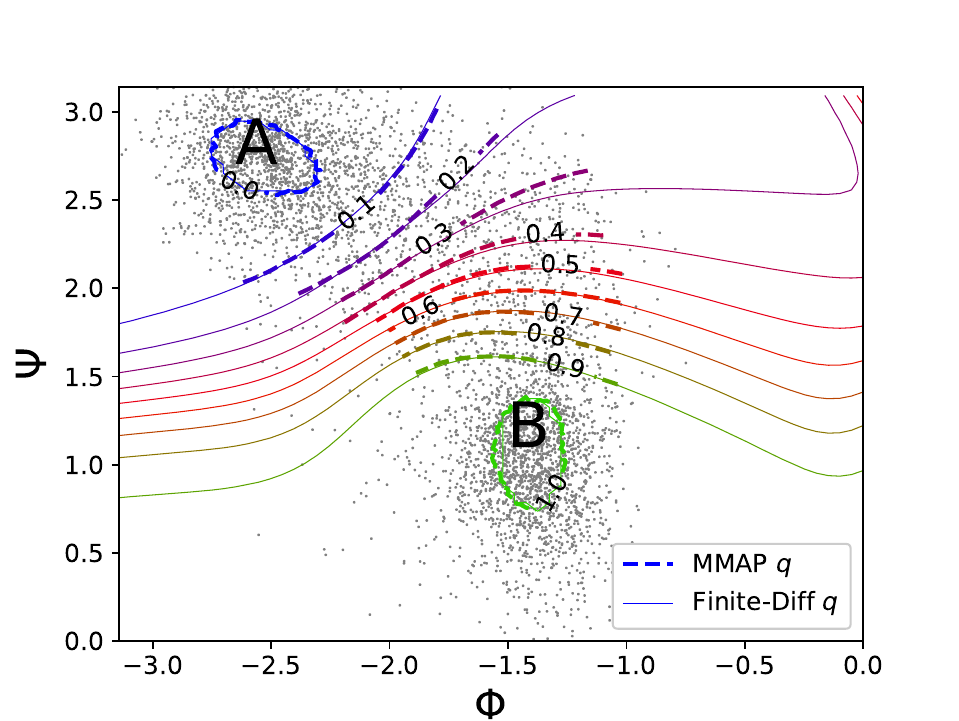}
(b)\includegraphics[width=0.45\textwidth]{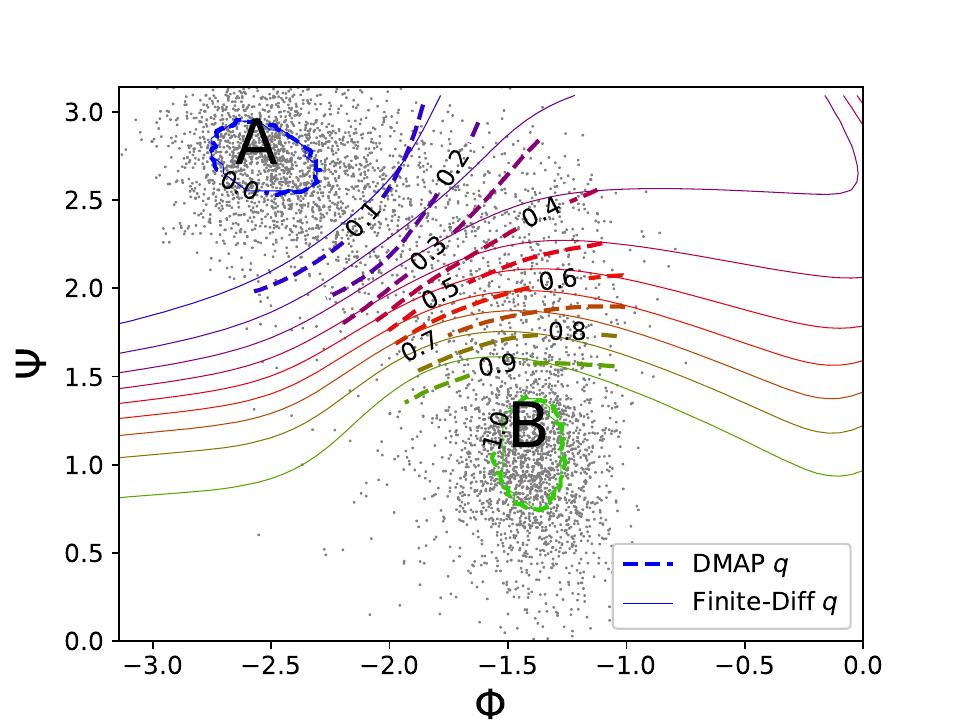}
\caption{
    Level sets for the approximate committor functions obtained from {\tt mmap} (a)
    and {\tt dmap} (b) on the point cloud (grey dots), with $A$ as the reactant region and $B$
    the product region. The dotted lines represent the committor level sets obtained by {\tt mmap} (a) and {\tt dmap} (b), while the solid lines depict  the committor level sets obtained by the
    finite-difference method.
}
\label{fig:c5c7_committor}
\end{center}
\end{figure}

As we have explained in Section \ref{sec:tpt} the committor allows us to compute the reactive current and the transition rate.
The calculation of the reactive current and the reaction rate is detailed in  \ref{appendix:current}.
The reactive currents computed using the {\tt mmap} and {\tt dmap} committors, respectively, are visualized in Figure~\ref{fig:ADcurr} (a) and (b). Notably, the intensity for the respective currents differs by an order of magnitude.
The corresponding reaction rates for {\tt mmap} and {\tt dmap} are, respectively, $\nu_{AB} = 0.092 \times 10^{-12}\text{s}^{-1}$ and  $\nu_{AB}=0.31 \times 10^{-12}\text{s}^{-1}$. 
To verify the rate, we ran 10 long trajectories and for each calculated the transition rate as the ratio of the number of transitions from $A$ to $B$ over the elapsed time. The mean rate over the trajectories is 
$\nu_{AB} = 0.093 \times 10^{-12}\text{s}^{-1}$ (standard deviation $0.003 \times 10^{-12}\text{s}^{-1})$ which is very close to the {\tt mmap} rate and notably differs from the {\tt dmap} rate. 
\begin{figure}[htbp]
\begin{center}
(a)\includegraphics[width=0.45\textwidth]{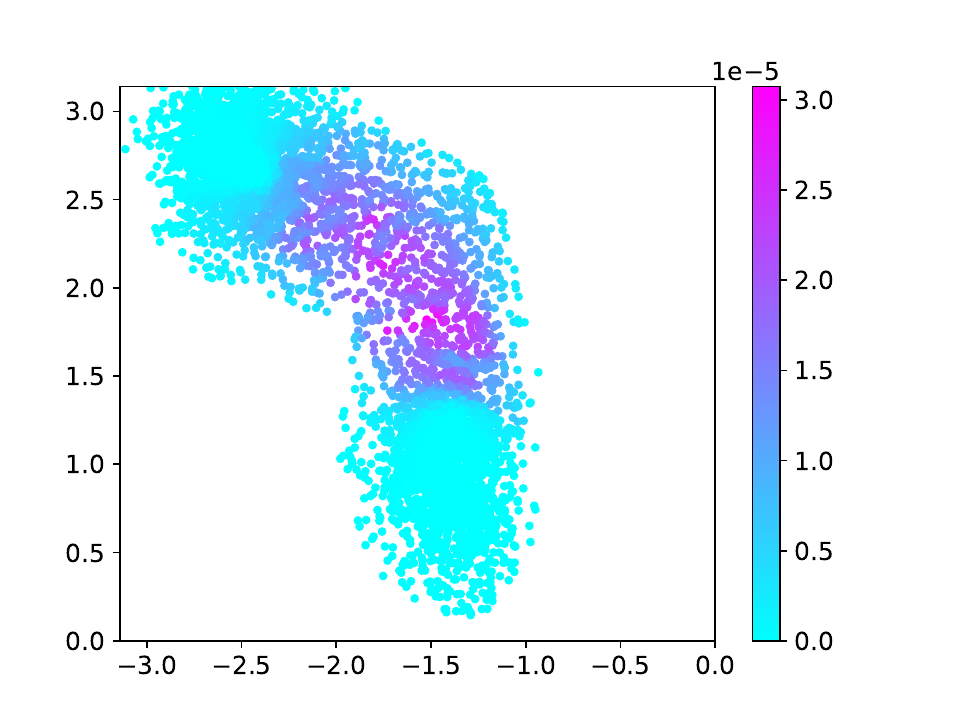}
(b)\includegraphics[width=0.45\textwidth]{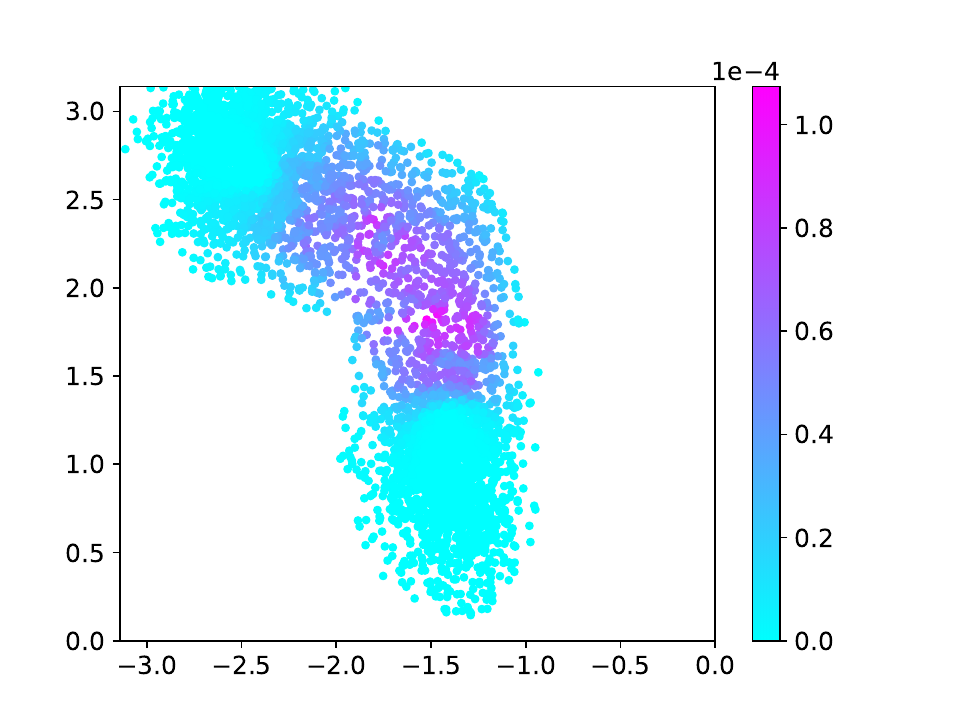}
\caption{ The intensity of the reactive current computed using the {\tt mmap} (a) and {\tt dmap} (b) committors.
}
\label{fig:ADcurr}
\end{center}
\end{figure}

%
%

\subsection{Transitions between the trapezoid and the hexagon in Lennard-Jones 7 in 2D}
\label{sec:LJ}

\begin{figure}[htbp]
\begin{center}
\includegraphics[width=0.8\textwidth]{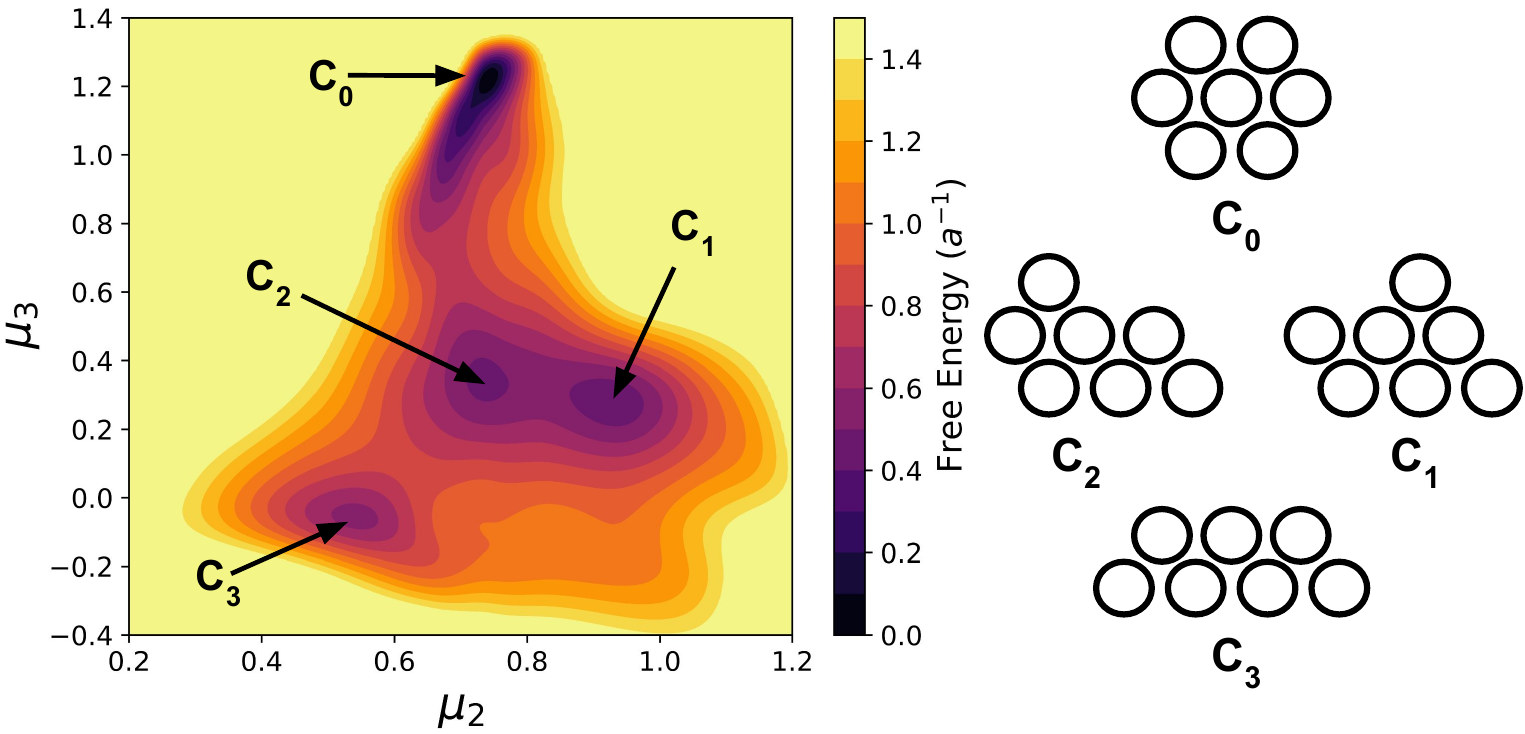}
\caption{ Free energy surface of LJ7 system with respect to 2nd and 3rd moment of coordination numbers CVs. 
The four minima  C$_k$, $k=0,1,2,3$, are marked in the free energy plot and depicted on the right.
}
\label{fig:LJ7_energy_clusters}
\end{center}
\end{figure}

The cluster of seven 2D particles interacting according to the Lennard-Jones pair potential
$$
V_{pair}(r) = 4a\left[ \Big(\frac{\sigma}{r}
\Big)^{12} - \Big(\frac{\sigma}{r}
\Big)^{6}\right]
$$
where $\sigma>0$ and $a>0$ are parameters controlling, respectively, range and strength of interparticle interaction,
has been another benchmark problem in chemical physics~\cite{ coifman2008diffusion, dellago1998efficient,wales2002discrete,passerone2001action}.
If the particles are treated as indistinguishable, the potential energy surface 
$$
V(x) = \sum_{i<j}V_{pair}(\|x_i-x_j\|),\quad 1\le i,j\le 7,
$$
has four local minima denoted by $C_0$ (hexagon), $C_1$ (capped parallelogram 1), $C_2$ (capped parallelogram 2), 
and $C_3$ (trapezoid) -- see Figure~\ref{fig:LJ7_energy_clusters}.

\subsubsection{Choosing collective variables}
Following~\cite{tribello2010self, tsai2019reaction}, we chose the 2nd and 3rd central moments of the
distribution of coordination numbers as collective variables (CVs). These CVs allow us to separate all four minima in a 2D space.
The coordination number of particle $i$, $1\le i\le 7$, is a smooth function approximating the number of nearest neighbors of $i$:
\begin{equation}
\label{eqn:cnum}
c_i(x) = \sum_{j \neq i} \frac{1 - \big(\frac{r_{ij}}{1.5 \sigma} \big)^{8}}{
1 - \big(\frac{r_{ij}}{1.5 \sigma} \big)^{16}},\quad{\rm where}\quad r_{ij}:=\|x_i-x_j\|.
\end{equation}
Let us elaborate on it. The interparticle distance minimizing $V_{pair}(r)$ is $r^{\ast}=2^{1/6}\sigma$. 
We would like to treat particles as nearest neighbors if the distance between them is  close to $r^{\ast}$. If four particles arranged into a square, 
the diagonal particles at distance $r^{\ast}\sqrt{2}\approx 1.5874\sigma$ should be ``not quite" nearest neighbors. Particles at distance $2r^{\ast}$ should not count as nearest neighbors. 
Normalizing the distance to $1.5\sigma$ 
in  \eqref{eqn:cnum} makes the desired distinction. Indeed, we have:
$$
 \frac{1 - \big(\frac{r^{\ast}}{1.5 \sigma} \big)^{8}}{
1 - \big(\frac{r^{\ast}}{1.5 \sigma} \big)^{16}} \approx 0.91\sim 1,\quad
 \frac{1 - \big(\frac{r^{\ast}\sqrt{2}}{1.5 \sigma} \big)^{8}}{
1 - \big(\frac{r^{\ast}\sqrt{2}}{1.5 \sigma} \big)^{16}} \approx 0.39,\quad
 \frac{1 - \big(\frac{2r^{\ast}}{1.5 \sigma} \big)^{8}}{
1 - \big(\frac{2r^{\ast}}{1.5 \sigma} \big)^{16}} \approx 0.04\sim 0.
$$
The $kth$ central moment of $c_i(x)$ is
\begin{equation}
\mu_k(x):=\frac{1}{7}\sum_i (c_i(x) - \bar{c}(x))^{k},\quad{\rm where}\quad \bar{c}(x) = \frac{1}{7}\sum_j (c_j(x)).
\end{equation}
The moments $\mu_k$ are invariant with respect to permutation of particles, which is important as the particles are indistinguishable.
 The space of the chosen collective variables $(\mu_2,\mu_3)$ is $\R^2$.

\subsubsection{Obtaining data}

\begin{figure}[htbp]
\begin{center}
\includegraphics[width=0.6\textwidth]{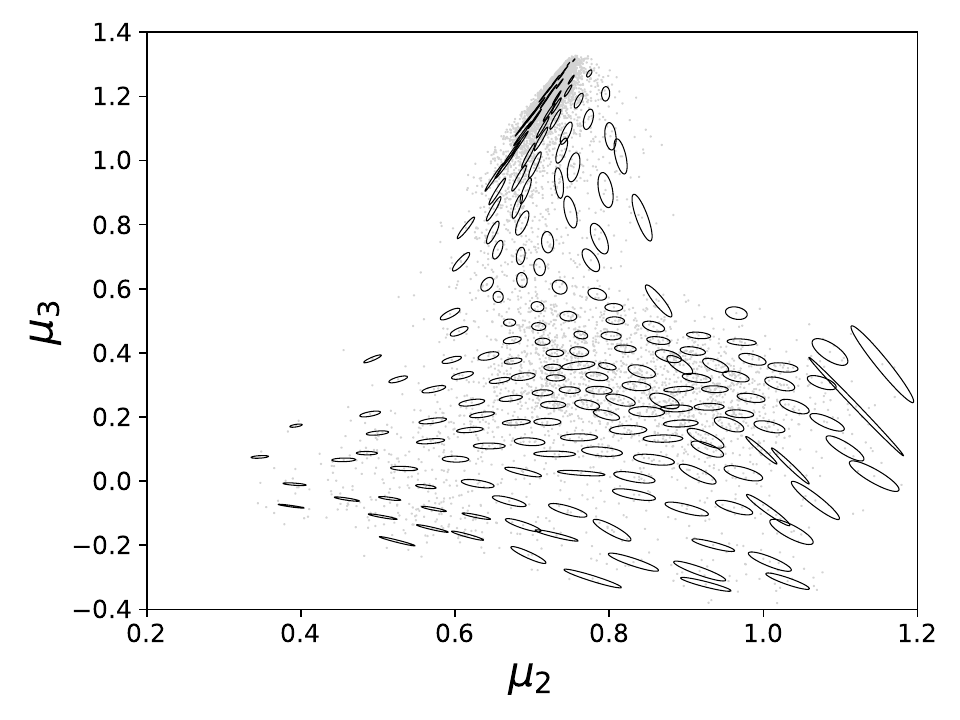}
\caption{Ellipses corresponding to the principal components of the estimated diffusion matrices $M(x_j)$ for a subsampled set of the LJ7 dataset depicted with faint grey dots. 
}
\label{fig:LJ7_diffusion_tensors}
\end{center}
\end{figure}
We set the temperature for the simulation to $0.2 \frac{k_bT}{a}$ and used a
Langevin thermostat with relaxation time $0.1\sqrt{\frac{a}{m\sigma^2}}.$ To
prevent clusters from evaporating, we imposed restraints to
keep the atoms from moving further than $2\sigma$ from the center of mass from
the cluster. Then we simulated the trajectory at timestep
$0.005\sqrt{\frac{a}{m\sigma^2}}$ for $10^{7}$ steps using the velocity Verlet
algorithm as implemented in the PLUMED software~\cite{plumed2}. For use with
diffusion maps, we subsampled the trajectory at regular intervals of time
to obtain $7500$ data points. 
\begin{figure}[htbp]
\begin{center}
(a)\includegraphics[width=0.45\textwidth]{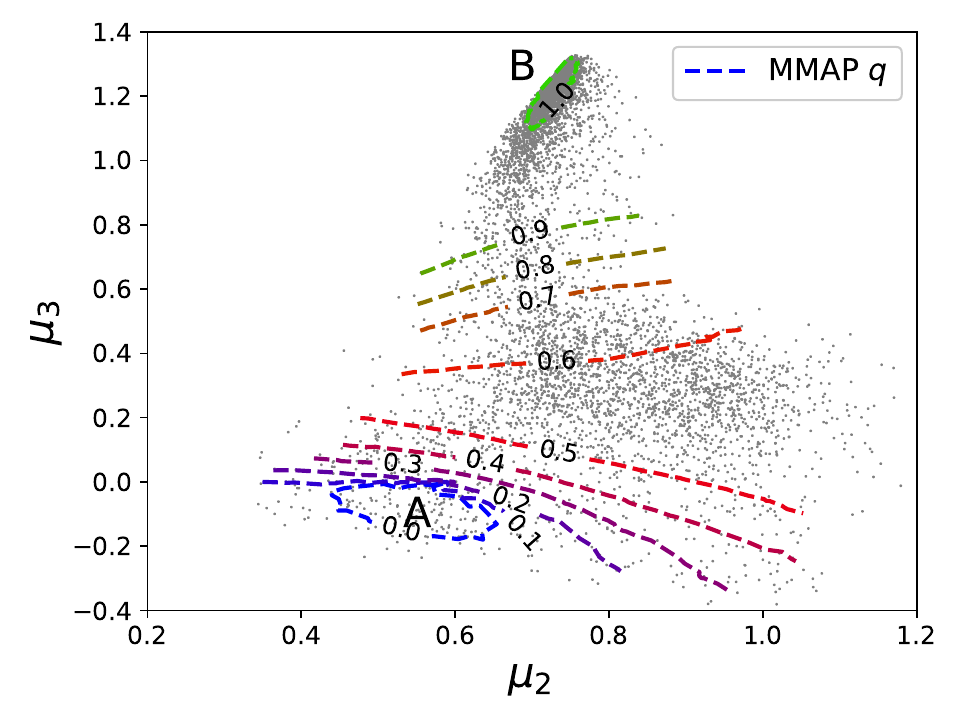}
(b)\includegraphics[width=0.45\textwidth]{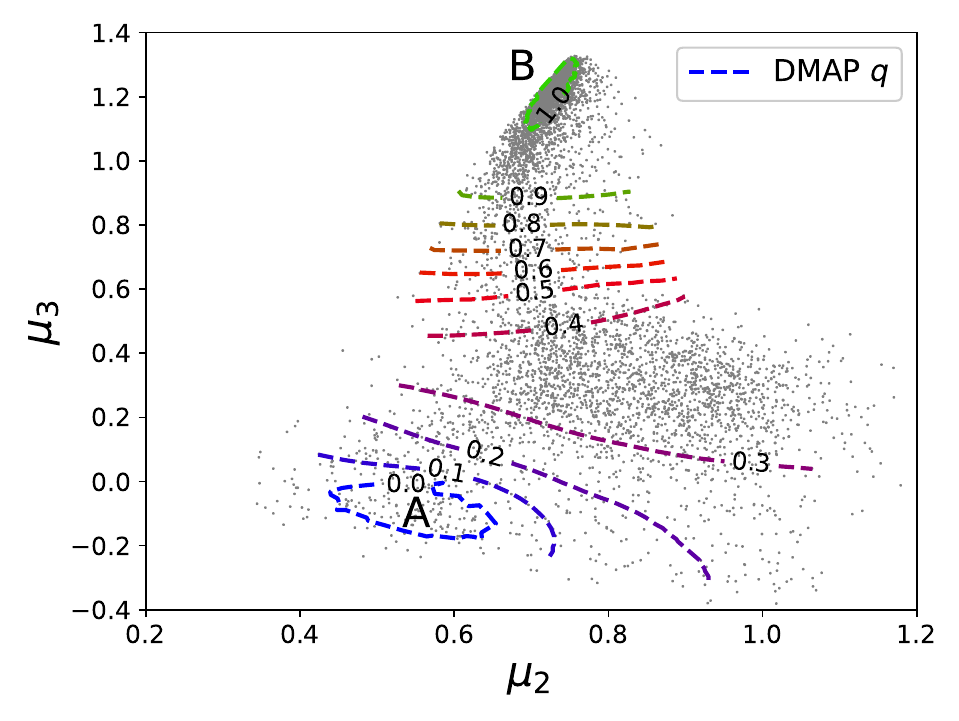}
\caption{ 
Level sets for the {\tt mmap} (a) and {\tt dmap} (b) committors are depicted with dashed lines. The set of data points is shown with grey dots.
}
\label{fig:LJ7_committor}
\end{center}
\end{figure}
The diffusion matrices obtained as described in~\cite{2006string} {and~\ref{sec:diffusion_estimation}}, and are visualized in
Figure~\ref{fig:LJ7_diffusion_tensors}. The reactant set $A$ and product
set $B$ are chosen to be the energy minima $C_3$ and $C_0$ respectively. This
choice was motivated by the fact that $C_3$ and $C_0$ are the most distant pair of metastable states, separated by the highest potential energy barrier~\cite{dellago1998efficient,wales2002discrete},
and 
connected with a wide transition channel passing through the basins of  $C_1$
and $C_2$. It is worth noting that such a situation where the region between two
stable states of interest is interspersed with other metastable states is quite
common in practical situations~\cite{martini2017variational}. 
Furthermore, transitions between $C_0$ and $C_3$ occur very infrequently, and 
it takes much longer time to accumulate statistics for them in numerical simulation than for the transition in alanine dipeptide considered in Section~\ref{sec:aladip}.

The scaling parameters $\epsilon$ for {\tt mmap} and {\tt dmap} were set, respectively, to
\begin{align}
\epsilon&= \max_{i} \min_{j\neq i} s(i,j),\quad{\rm where}\label{eqn:heuristic}\\
s(i,j)&:= \frac{1}{2}(x_i -x_j)^{\top}(M^{-1}(x_i) + M^{-1}(x_j))(x_i - x_j)\quad\text{for {\tt mmap}},\\
s(i,j)&:= ||x_i - x_j||_2^2\quad\text{for {\tt dmap}}.
\end{align}
This simple procedure for choosing $\epsilon$ worked remarkably well.

\subsubsection{Results and validation}
\label{sec:LJ_results}
The level sets of the {\tt mmap} and {\tt dmap} committors are shown in Figure \ref{fig:LJ7_committor}.
Notably, these  committors significantly differ from each other. In particular, the $q=0.5$
level sets for the {\tt mmap} and {\tt dmap} committors
lie on opposite sides of the dynamical trap surrounding the basins of $C_1$ and $C_2$ minima.
{ The reactive currents  computed using the {\tt mmap} and {\tt dmap} committors respectively, are displayed in Figure~\ref{fig:LJ7_current} (a) and (b). 
}
\begin{figure}[htbp]
\begin{center}
(a)\includegraphics[width=0.45\textwidth]{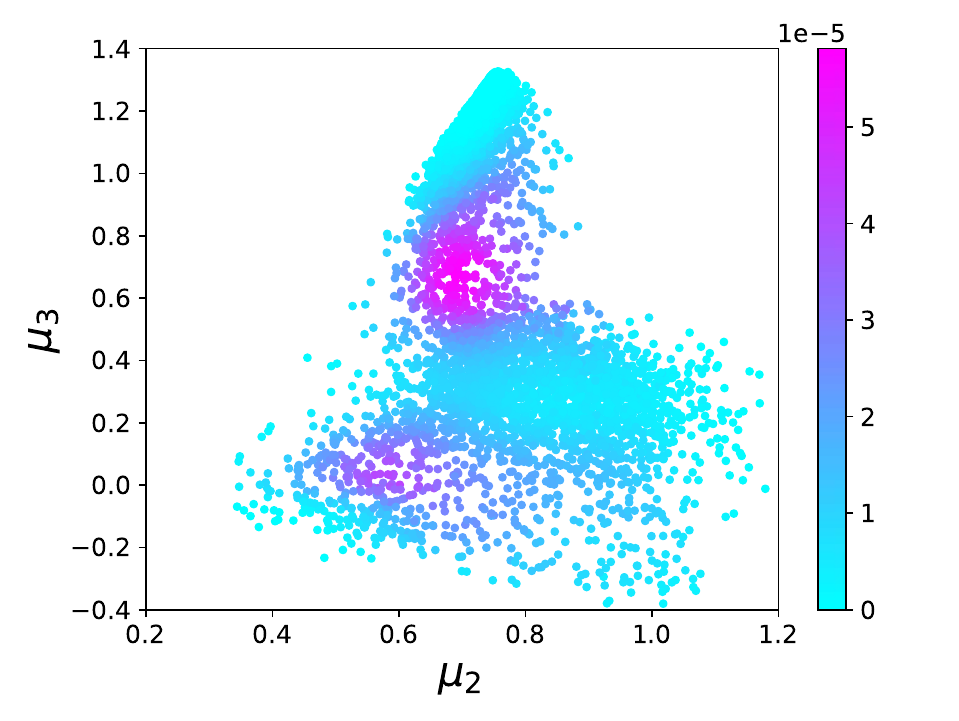}
(b)\includegraphics[width=0.45\textwidth]{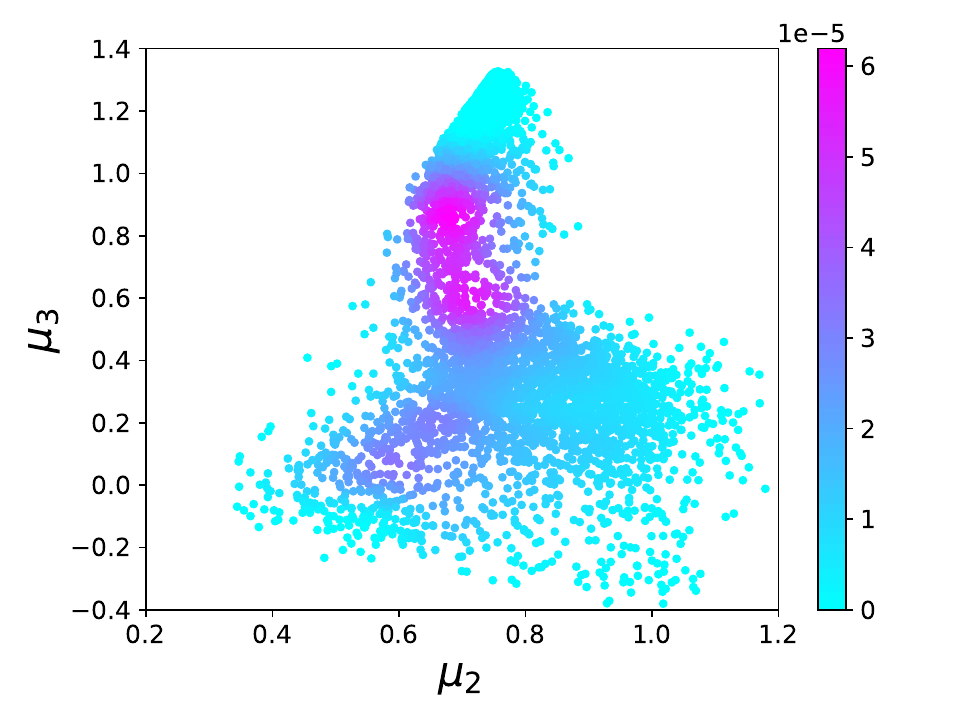}
\caption{ 
The intensity of the reactive current obtained with the {\tt mmap} (a) and {\tt dmap} (b) committors. 
}
\label{fig:LJ7_current}
\end{center}
\end{figure}

To validate our results and determine which of the {\tt mmap} or {\tt dmap} committors is more accurate, we performed committor
analysis \cite{2006string,peters2016reaction,geissler1999kinetic}, a common statistical
validation technique for committors in collective variables. 
Committor analysis checks a particular committor
level set by using the definition of the committor at $x$ as the probability that a stochastic trajectory starting at $x$ first reaches $B$ rather than $A$. 
We verified the most important level set $q=0.5$, the \emph{transition state}.
We sampled a set of  $N_{pt}=1000$ points $x_j$ along this level set and launched an ensemble of $N_e = 200$
trajectories from each of them. 
For each $x_j$, we counted the number of trajectories $N_B$ that reached first C$_0$ rather than C$_3$ and denoted the ratio $\sfrac{N_B}{N_e}$
by $p_B(x_j)$. 
We then plotted a histogram with each bin defined by a $p_B$ value and counts
determined by the number of selected points in the level set $q=0.5$ with that $p_B$ value normalized by $N_{pt}$ {(Figure~\ref{fig:LJ7_committor_analysis})}.
A well-approximated $q=0.5$ level set should have a unimodal histogram
with a sharp peak at $p_B=0.5$. We see that the distribution for {\tt mmap} peaks at $0.5$ as expected, while the {\tt dmap} distribution peaks at $0.75,$ missing the correct statistical behavior by a large margin.
Therefore, we conclude that {\tt mmap} produces a good approximation for the committor while {\tt dmap} gives a qualitatively wrong result.

\begin{figure}[htbp]
\begin{center}
\includegraphics[width=0.6\textwidth]{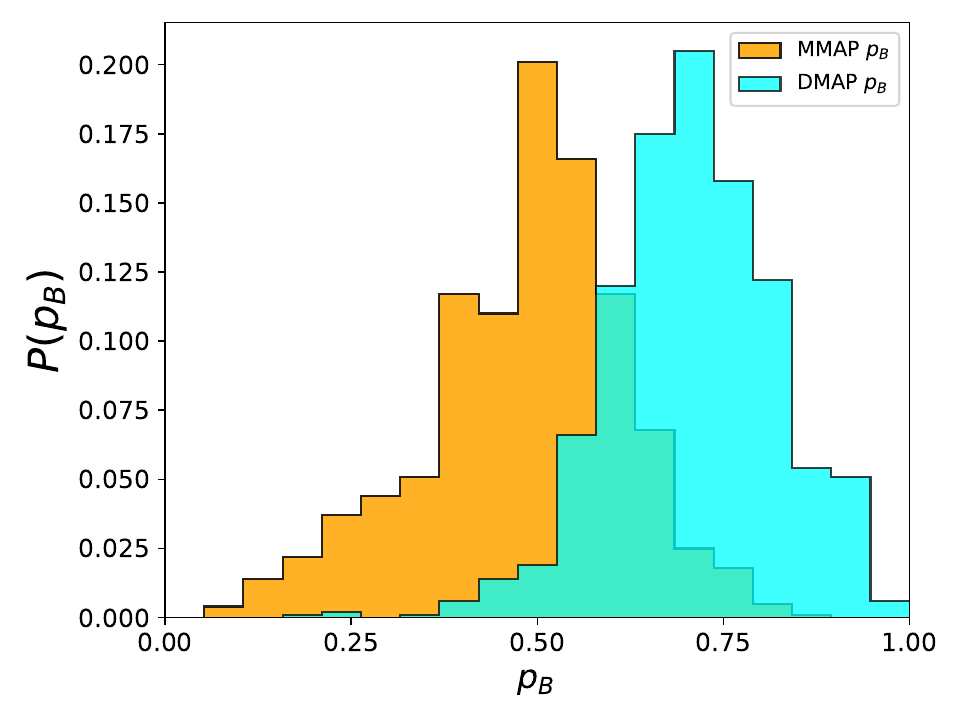}
\caption{Committor analysis of the $q=0.5$ level sets from Figure~\ref{fig:LJ7_committor}, with orange corresponding to {\tt mmap} and cyan corresponding to {\tt dmap}.
    }
\label{fig:LJ7_committor_analysis}
\end{center}
\end{figure}

%
%

\section{Conclusion}
\label{sec:conclusion}
The main conclusion of this work is that the Mahalanobis diffusion map algorithm ({\tt mmap}) is a provably correct, robust and reliable tool for computing the committor 
in collective variables discretized to a point cloud of data generated by MD simulations. The dynamics in collective variables is governed by a reversible
SDE with anisotropic and position-dependent diffusion matrix $M(x)$. 
The Mahalanobis kernel proposed in \cite{singer2008} accurately captures this anisotropy regardless of whether 
$M(x)$ is decomposable or not into a product $J(x)J^{\top}(x)$  where $J(x)$ is the Jacobian matrix for some diffeomorphism. 

Specifically, we have calculated the limiting family of differential operators converged to by the $\alpha$-indexed {\tt mmap} family of matrix operators, where convergence is with respect to the number of data points tending to infinity and scaling parameter $\epsilon$ tending to 0.
If $\alpha=\sfrac{1}{2}$, the limiting operator is the generator for the overdamped Langevin 
SDE in collective variables. In our derivation, we have discarded the key assumption of \cite{singer2008} that $M(x)$ is associated with a diffeomorphism.

We have chosen two benchmark chemical physics systems as test problems: transitions in alanine dipeptide and rearrangement of an LJ7 cluster of 2D particles.
On these examples, we have demonstrated that {\tt mmap} is easy to implement and gives good results for any reasonable choice of the scaling parameter epsilon. 
We have validated our results by comparing the committors  computed with {\tt mmap} to the one obtained using a traditional finite difference method or by conducting committor analysis.
We have contrasted the committor by {\tt mmap} with the one by the diffusion map with isotropic Gaussian kernel and shown that the latter can  
lead to a wrong placement of the transition state and highly inaccurate estimate for the reaction rate.

In the current setting, {\tt mmap} has a significant limitation: it requires the input data to be sampled from the invariant probability density. This prevents us from
using enhanced sampling techniques such as temperature acceleration~\cite{maragliano2006temperature} and metadynamics~\cite{valsson2016} which are standard 
techniques applied to promote transitions between metastable states in MD simulations. 
We plan to address this problem in our future work.

\section{Acknowledgements}
We thank Prof. Y. Kevrekidis for encouraging us to look into Mahalanobis diffusion maps and for providing valuable feedback on this manuscript. { We thank Prof. T. Berry for a valuable discussion on whether every symmetric positive definite matrix function is associated with a variable change. We are grateful to the anonymous reviewers for their thorough reviews and helpful suggestions. }
This work was  partially supported by NSF CAREER grant DMS-1554907 (MC), AFOSR MURI grant FA9550-20-1-0397 (MC),
and  NSF CAREER grant CHE-2044165 (PT). 
Computing resources Deepthought2, MARCC and XSEDE (project TG-CHE180053) were used to generate molecular simulation data.

 \appendix
\setcounter{equation}{0}
\renewcommand{\theequation}{\Alph{section}-\arabic{equation}}
    \setcounter{lemma}{0}
    \renewcommand{\thelemma}{\Alph{section}\arabic{lemma}}

{

\section{The free energy and the diffusion tensor}
\label{sec:diffusion_estimation}
The free energy $F(x)$ and the diffusion matrix $M(x)$ in SDE \eqref{eqn:cvsde} are defined, respectively, as follows:
\begin{align}
\label{eqn:free_energy_formula}
    F(x) = - \beta^{-1} \ln \left(\int_{\R^{m}} Z_y^{-1} e^{-\beta V(y)} \prod\limits_{l = 1}^{d}\delta(\theta_{l}(y) - x_l) dy\right),
\end{align}
\begin{align}
    \label{eqn:diffusion_tensor_formula}
M(x) = e^{\beta F(x)} \int_{\R^{m}} J(y) J^{\top}(y) Z_{y}^{-1} e^{-\beta V(y)}\prod\limits_{l = 1}^{d}\delta(\theta_{l}(y) - x_l) dy.
 \end{align}

In \eqref{eqn:diffusion_tensor_formula}, $J(x)$ is the Jacobian matrix whose entries are 
$$
J_{ij}(y) = \frac{\partial\theta_i(y)}{\partial y_j} \qquad 1 \le i \le d, \quad 1 \le j \le m.
$$
To apply the {\tt mmap} algorithm, we do not need to know the free energy. However, we do need evaluate the diffusion matrix $M(x)$ at the data points. 
A method for estimating the diffusion matrix $M(x)$ of~\eqref{eqn:diffusion_tensor_formula} was described in~\cite{2006string}. Here we outline it for the reader's convenience.

First, we approximate the distribution $\prod_{l=1}^d \delta(\theta_l(y) - x_l)$ with a Gaussian. 
We fix $x \in \R^{d}$, choose a large spring constant $\kappa>0$, and consider a constrained system with the ``extended potential'' given by
\begin{align}
U(y;\kappa,x) = V(y) + \frac{\kappa}{2}||\theta(y) - x||^2,
\end{align}
evolving according to the overdamped Langevin dynamics
\begin{align}\label{eqn:constr_langevin}
dy_t &= -\nabla U(y_t;\kappa,x) dt + \sqrt{2 \beta^{-1}} dW_t \nonumber \\
&=\left[-\nabla V(y) - \kappa J(y)^{\top}(\theta(y) - x)\right]dt + \sqrt{2\beta^{-1}}dW_t.
\end{align}
The restrained dynamics~\eqref{eqn:constr_langevin} has stationary distribution 
$$
\rho(y;\kappa,x):= Z(\kappa,x)^{-1}e^{-\beta U(y;\kappa,x)}\quad{\rm where}\quad Z(\kappa,x):= \int_{\R^d} e^{-\beta U(y;\kappa,x)} dy,
$$
with limiting distribution 
$$\lim\limits_{\kappa \to \infty} \int_{\R^d} f(y) \rho(y;\kappa, x) dy = \int_{\R^d} f(y) Z_y^{-1} e^{-\beta U(y)}\prod\limits_{l=1}^d \delta(\theta_l(y) - x_l) dy $$
Next, we generate trajectory data $\{y_{t_i}\}_{i=1}^n$ for the restrained dynamics \eqref{eqn:constr_langevin}. These data enable us to estimate the conditional expectation for an arbitrary function $f$ as follows:
\begin{align}
\lim\limits_{\kappa \to \infty}\lim\limits_{n \to \infty} \frac{1}{n} \sum\limits_{i=1}^n f(y_{t_i}) 
&= \lim\limits_{\kappa \to \infty} \lim\limits_{n \to \infty} \frac{1}{n\Delta t} \int_{0}^{n\Delta t} f(y_t) \rho(y_t;\kappa,x) dt \nonumber \\
&= e^{\beta F(x)} \int_{\R^{n}} f(y) Z_y^{-1} e^{-\beta V(y)} \prod\limits_{l=1}^d \delta(\theta_l(y) - x_l) dy  \nonumber \\
&=  \E[f \vert \theta(y) = x].
\end{align}
In particular, the diffusion tensor $M(x)$ for the collective variables 
$x = \theta(y)$ is estimated according the formula:
\begin{equation}
    \label{eqn:Mcompute}
    M_{ij}(x)\approx \frac{1}{n}\sum_{k=1}^n\sum_{l=1}^m \frac{\partial\theta_i(y_{t_k})}{\partial y_l}\frac{\partial\theta_j(y_{t_k})}{\partial y_l}.
\end{equation}
The mean force, i.e., the gradient of the free energy $\nabla F(x)$, also can be estimated in a similar manner:
\begin{equation}
\label{eqn:meanforcecompute}
\nabla F(x) = \frac{\kappa}{n}\sum_{i=1}^n\left(x - \theta(y_{t_i})\right).
\end{equation}
We evaluate $M(x)$ by \eqref{eqn:Mcompute} in both applications presented in this work. This procedure is an outgrowth of well-established uses for constrained dynamics within the molecular dynamics community, particularly in fundamental works for computing free energy differences~\cite{carter1989constrained,kirkwood1935statistical} and
position-dependent friction~\cite{straub1987calculation}.

For a general diffusion matrix not necessarily of form~\eqref{eqn:diffusion_tensor_formula} or for when the Jacobian $J(y)$ is not available, one can utilize local covariances as described in~\cite{singer2008, banisch2020,talmon2013empirical,singer2009detecting, peter2020}. These approaches utilize that for~\eqref{eqn:cvsde},
\[M(x) = \lim\limits_{\Delta t \to 0}\frac{\mathbb{E}\big[(x_{t + \Delta t} - x_t)(x_{t+ \Delta t} - x_t)^{\top} 
\vert x_t = x\big]}{2 \beta^{-1}\Delta t}, \]
and derive estimators of $M(x)$ by approximating the right-hand side
through short simulation bursts initiated at $x$ or from small neighborhoods of the trajectory near $x.$  


}

\section{Proof of Theorem \ref{thm:lem6}}
\label{appendix:not_every_diffusion}

We will need two auxiliary lemmas.
The first lemma gives a necessary condition for a matrix-function to be Jacobian. 
\begin{lemma}
\label{thm:lem4}
Let $z = f(x)$ be a twice continuously differentiable coordinate change $f:\R^d\rightarrow\R^d$ with Jacobian matrix
\begin{equation}
\label{eqn:Jac1}
J(x) = \left[\begin{array}{ccc}\frac{\partial f_1}{\partial x_1}&\cdots&\frac{\partial f_1}{\partial x_d}\\
\vdots&&\vdots\\\frac{\partial f_d}{\partial x_1}&\cdots&\frac{\partial f_d}{\partial x_d}\end{array}\right].
\end{equation}
Then, the entries of $J$ satisfy
\begin{equation}
\label{eqn:Jac2}
 \frac{\partial J_{ij}}{\partial x_k} = \frac{\partial J_{ik}}{\partial x_j}\quad\forall 1\le i,j,k\le d,~ j\neq k.
\end{equation}
\end{lemma}
\begin{proof} 
Indeed, the left and right-hand side of \eqref{eqn:Jac2} are the mixed partials that are equal as they are continuous:
$$
 \frac{\partial J_{ij}}{\partial x_k}  = \frac{\partial^2 f_i}{\partial x_k\partial x_j}  = 
  \frac{\partial^2 f_i}{\partial x_j\partial x_k} = \frac{\partial J_{ik}}{\partial x_j}.
  $$
\end{proof}

The second lemma shows that any decomposition $M=AA^\top$ of a symmetric positive definite matrix
relates to $M^{1/2}$ via an orthogonal transformation.
\begin{lemma}
\label{thm:lem5}
Let $M$ be a symmetric positive definite matrix, and $A$ be any matrix such that
$M = AA^\top$.
Then, there exists an orthogonal transformation $O$ such that $A = M^{1/2}O$.
\end{lemma}
\begin{proof}
We have: 
$$
M = M^{1/2}M^{1/2} = AA^\top.
$$
Multiplying this identity by $M^{-1/2}$ on the right and on the left we get:
$$
I = M^{-1/2}AA^\top M^{-1/2} = M^{-1/2}A\left(M^{-1/2}A\right)^\top.
$$
Hence $O:=M^{-1/2}A$ is orthogonal.
Therefore, $A = M^{1/2}O$ as desired.
\end{proof}

{
\begin{proof} (Proof of Theorem \ref{thm:lem6}.)
First we prove that \emph{\eqref{eqn:harmonic} is necessary for the existence of decomposition \eqref{eqn:Mdec}}.
We observe that $M^{1/2}(x,y) = m(x,y)I_{2\times 2}$ is a Jacobian of a vector-function $f:\Omega\rightarrow\mathbb{R}^2$ if and  only if $m(x,y)$ is constant. Indeed, condition \eqref{eqn:Jac2} applied to $M^{1/2}(x,y) = m(x,y)I_{2\times 2}$ reduces to $m_y = 0$ and $m_x = 0$. Note that any constant function is harmonic.

Suppose that $m(x,y)$ is not constant. In this case, by Lemma \ref{thm:lem5}, if $M$ admits decomposition \eqref{eqn:Mdec} then $J(x,y)$ must be of the form $M^{1/2}(x,y)O(x,y)$ for some orthogonal matrix $O(x,y)$. There are two families of orthogonal $2\times 2$ matrix functions:
\begin{equation}
\label{eqn:O}
O(x,y) =  \left[\begin{array}{cc}\cos\phi(x,y) &\sin\phi(x,y)\\-\sin\phi(x,y)&\cos\phi(x,y)\end{array}\right]~~{\rm and}~~
O(x,y) =  \left[\begin{array}{cc}\cos\phi(x,y) &\sin\phi(x,y)\\\sin\phi(x,y)&-\cos\phi(x,y)\end{array}\right].
\end{equation}
Hence, $M^{1/2}(x,y)O(x,y)$ is of the form
$$
m(x,y)\left[\begin{array}{cc}\cos\phi(x,y) &\sin\phi(x,y)\\-\sin\phi(x,y)&\cos\phi(x,y)\end{array}\right]~~{\rm or}~~
m(x,y)\left[\begin{array}{cc}\cos\phi(x,y) &\sin\phi(x,y)\\\sin\phi(x,y)&-\cos\phi(x,y)\end{array}\right].
$$
Condition \eqref{eqn:Jac2} applied to $M^{1/2}(x,y)O(x,y)$ requires the following equalities to hold:
\begin{align*}\frac{\partial}{\partial y}\left(m\cos\phi\right) &= \frac{\partial}{\partial x}\left(m\sin\phi\right),\\ 
\frac{\partial}{\partial y}\left(m\sin\phi\right) &= -\frac{\partial}{\partial x}\left(m\cos\phi\right).
\end{align*}
Performing differentiation, we get:  
\begin{align*}
m_y\cos\phi -m\phi_y\sin\phi &=m_x\sin\phi + m\phi_x\cos\phi\\
m_y\sin\phi + m\phi_y\cos\phi &=-\left(m_x\cos\phi - m\phi_x\sin\phi\right)
\end{align*}
Regrouping the terms, we obtain: 
\begin{align}
m_y\cos\phi - m_x\sin\phi &= m\left[\phi_x\cos\phi + \phi_y\sin\phi\right] \label{eqn:aux1} \\
-m_y\sin\phi - m_x\cos\phi &= m\left(\phi_y\cos\phi - \phi_x\sin\phi\right).\label{eqn:aux2}
\end{align}
The last set of identities can be rewritten in a matrix form:
\begin{equation}
\label{eqn:matrix_form1}
\left[\begin{array}{cc}\cos\phi &\sin\phi\\
-\sin\phi&\cos\phi\end{array}\right]
\left[\begin{array}{r}m_y\\-m_x\end{array}\right] =
m\left[\begin{array}{cc}\cos\phi &\sin\phi\\
-\sin\phi&\cos\phi\end{array}\right]
\left[\begin{array}{r} \phi_x\\ \phi_y\end{array}\right] 
\end{equation}
The matrix in \eqref{eqn:matrix_form1} is orthogonal. Hence, multiplying \eqref{eqn:matrix_form1} by its transpose we get:
\begin{equation}
\label{eqn:harmonic2}
\frac{m_y}{m} \equiv [\log m]_y = \phi_x,\quad -\frac{m_x}{m} \equiv  -[\log m]_x= \phi_y.
\end{equation}
The fact that the mixed partials of $\phi$ must be equal, implies that 
\begin{equation}
\label{eqn:lap0}
[\log m]_{xx} + [\log m]_{yy} = 0.
\end{equation}
This completes the proof that \eqref{eqn:harmonic} is necessary for the existence of decomposition \eqref{eqn:Mdec}.

Next, we prove that \emph{\eqref{eqn:harmonic} is sufficient for the existence of decomposition \eqref{eqn:Mdec} if $\Omega$ is simply connected.} This immediately follows from the theorem of calculus saying that if the components of a two-dimensional continuously differentiable vector field $[p,q]^\top$ satisfy $p_y=q_x$ in a simply connected domain $\Omega$ then this vector field is conservative. Indeed,  we define the vector field $[p,q]^\top$  by $p = [\log m]_y$ and $q=-[\log m]_x$. This vector field satisfies the condition $p_y=q_x$ as $\log m$ is harmonic. We choose a point $(x_0,y_0)\in\Omega$, fix it, and for any other point $(x,y)\in\Omega$ choose a path $\gamma\subset\Omega$ from $(x_0,y_0)$ to $(x,y)$ and integrate the vector field $(p,q)$ along it (see the integral in the right-hand side of equation \eqref{eqn:phidef} below). We claim that the value of this integral is independent of the path $\gamma$ from $(x_0,y_0)$ to $(x,y)$. Indeed, if $\gamma'$ is some other path connecting these points, then we define a closed contour by reversing the path $\gamma'$ and apply Green's theorem 
\begin{equation}
\label{eqn:green}
\oint_{C} pdx + qdy = \iint_{U}(p_y-q_x)dxdy,
\end{equation}
where $U$ is the region bounded by the simple closed contour $C$. Since $p_y-q_x$ is identically zero, the contour integral also must be zero. If the contour formed by the paths $\gamma$ and the reverse of $\gamma'$ is self-intersecting, we apply Green's theorem to all simple closed contours formed by these paths. Therefore, we can define a function $\phi(x,y)$ by
\begin{equation}
\label{eqn:phidef}
\phi(x,y) = \int_{\gamma}[\log m]_ydx - [\log m]_xdy
\end{equation}
where $\gamma\subset\Omega$ is any path from $(x_0,y_0)$ to $(x,y)$ set
$J(x,y) = m(x,y)O(x,y)$ where $O(x,y)$ is any orthogonal matrix function of the form \eqref{eqn:O} with $\phi$ defined by \eqref{eqn:phidef}.

Finally, we show that \emph{if $\Omega$ is not simply connected, the condition \eqref{eqn:harmonic} is not sufficient for the existence of decomposition \eqref{eqn:Mdec}}. We adapt the famous counterexample of a non-conservative vector field. Consider the vector field 
\begin{equation}
\left[\begin{array}{r}p\\q\end{array}\right]=\frac{a}{x^2+y^2}\left[\begin{array}{r}-y\\x\end{array}\right],
\end{equation}
where $a\neq 0$ is a constant,
with the property that $p_y = q_x$. It is smooth in $\mathbb{R}^2\backslash\{(0,0)\}$. Let 
\begin{equation}
\label{eqn:ex1}
m(x,y) = \left(x^2+y^2\right)^{-a/2}.
\end{equation}
It is easy to check that $[\log m]_y = p$ and $-[\log m]_x = q$ and hence $\log m$ is harmonic everywhere except for the origin. Let us choose $\phi(x,y)$ satisfying $\phi_x = p$ and $\phi_y = q$ to be\footnotemark[1]
\footnotetext[1]{Acknowledgement of lecture notes by E.~L. Lady: \\\href{http://www.math.hawaii.edu/~lee/calculus/potential.pdf}{http://www.math.hawaii.edu/$\sim$lee/calculus/potential.pdf}}
\begin{equation}
\label{eqn:ex2}
\phi(x,y) = \begin{cases}a\arctan\left(\sfrac{y}{x}\right),& x > 0\\
\frac{\pi a}{2},& x = 0\\
\pi a+a\arctan\left(\sfrac{y}{x}\right),& x < 0
\end{cases}
\end{equation}
The function $\phi$ is smooth everywhere except for the origin and the negative $y$-axis. It has a jump discontinuity of size $2\pi a$ along the negative $y$-axis. If $a\notin\mathbb{Z}$, the functions $\sin\phi$ and $\cos\phi$ will be discontinuous along the negative $y$-axis. Hence decomposition \eqref{eqn:Mdec} does not exist.  
\end{proof}
\begin{remark}
However, if $a\in\mathbb{Z}$ in \eqref{eqn:ex1}, the orthogonal matrices \eqref{eqn:O} with $\phi$ given by \eqref{eqn:ex2} are smooth everywhere except for the origin. In particular, if $a=1$, we have:
\begin{equation}
\label{eqn:ex3}
\frac{1}{\sqrt{x^2+y^2}}\left[\begin{array}{rr}x&y\\-y&x\end{array}\right]\quad{\rm and}\quad
\frac{1}{\sqrt{x^2+y^2}}\left[\begin{array}{rr}x&y\\y&-x\end{array}\right]
\end{equation}
\end{remark}

}

%
%

\section{Proof of Theorem \ref{thm:main-theorem}}
\label{appendix:main_proof}
We fix $r\ge0$, $x\in\R^d$, and a symmetric positive definite matrix $A\in\R^{d\times d}$.
Then $\mathcal{B}_{r}(x;A)$ denotes the ellipsoid
$$
\mathcal{B}_{r}(x;A):=\{y\in\R^d~|~(y-x)^\top A(y-x)\le r^2\}.
$$
The proof of Theorem \ref{thm:main-theorem} includes two technical lemmas.

\begin{lemma}
\label{thm:lem0}
Let $x\in\R^d$ be fixed, $\phi(x+a)$ be a function that grows not faster than a polynomial as $\|a\|\rightarrow\infty$, and 
$A$ be a positive definite matrix. 
Then, for all small enough $\epsilon>0$  and any $\mu\in(0,\sfrac{1}{2})$, we have:
\begin{equation}
I:=
\left|\int_{\R^d\backslash \mathcal{B}_{\epsilon^{\mu}}(x;A)}
e^{-\frac{1}{2\epsilon} (y-x)^\top A(y-x)}\phi(y) dy \right|
 \le \epsilon^{d/2}p(\epsilon^{2\mu-1})e^{-\frac{\epsilon^{2\mu-1}}{2}},\label{eq:lem0}
\end{equation}
where $p$ is a polynomial.
\end{lemma}
\begin{proof}
We implement two variable changes. First we introduce $z:=\epsilon^{-1/2}A^{1/2}(y-x)$ and then switch to
spherical coordinates in $\R^d$. We calculate:
\begin{align*}
I=&
\left|\int_{\R^d\backslash \mathcal{B}_{\epsilon^{\mu}}(x;A)}
e^{-\frac{1}{2\epsilon} (y-x)^\top A(y-x)}\phi(y) dy \right|\\
=&
\left|\int_{\R^d\backslash \mathcal{B}_{\epsilon^{2\mu-1}}(0;I)} e^{-\frac{1}{2} \|z\|^2}
\phi(x+\sqrt{\epsilon}A^{-1/2}z)\epsilon^{d/2}|A|^{-1/2}dz\right|
\end{align*}
Note that $2\mu - 1 < 0$, hence $\epsilon^{2\mu-1}\rightarrow\infty$ as $\epsilon\rightarrow 0$.
Further, since $\phi(x+a)$ grows not faster than a polynomial as $\|a\|\rightarrow\infty$, there are constants $C_1$ and $C_2$  and a positive integer $k$ such that
$$
\left|\left(\phi(x)+\sqrt{\epsilon} A^{-1/2}z\right)\right| \le C_1 + \sqrt{\epsilon}C_2\lambda_{\max}(A^{-1/2})\|z\|^k.
$$
Therefore, aiming at switching to spherical coordinates in $\R^d$, we write:
$$
\left|\|z\|^{d-1}\left(\phi(x)+\sqrt{\epsilon} A^{-1/2}z\right)\right| \le C_1\|z\|^{d-1} + 
\sqrt{\epsilon}C_2\lambda_{\max}(A^{-1/2})\|z\|^{k+d-1}\le
C\|z\|^{m},
$$
where $C$ is some constant, and $m$ is the smallest odd integer greater or equal to $d+k-1$.
Finally, switching to spherical coordinates and denoting the surface of  the $(d-1)$-dimensional unit sphere by $|S_{d-1}|$, 
we derive the desired estimate:
\begin{align*}
I & \le \frac{C\epsilon^{d/2}|S_{d-1}|}{|A|^{1/2}} 
\int_{\epsilon^{\mu-\sfrac{1}{2}} }^{\infty} e^{-\frac{r^2}{2}}r^mdr\\
=& \frac{C\epsilon^{d/2}|S_{d-1}|}{|A|^{1/2}} 
\int_{\frac{\epsilon^{2\mu-1}}{2} }^{\infty} e^{-t}t^{\frac{m-1}{2}}dt\\
=&  \epsilon^{d/2}p(\epsilon^{2\mu-1})e^{-\frac{\epsilon^{2\mu-1}}{2}},
\end{align*}
where the polynomial $p$ is obtained from  integrating by parts $\sfrac{(m-1)}{2}$ times and multiplying the result by $C|A|^{-1/2}$.
\end{proof}

\begin{lemma}
\label{thm:lem2}
Let $\mathcal{G}_{\epsilon}$ be an integral operator defined by
\begin{equation}
\label{eqn:Ge}
\mathcal{G}_{\epsilon}f(x) = \int_{\R^d} e^{-\frac{1}{4\epsilon}(x-y)^\top [M^{-1}(x) +M^{-1}(y)](x-y)}f(y) dy,
\end{equation}
where the matrix function $M$ and the scalar function $f$ satisfy Assumption \ref{Ass2}.
Let $\tilde{M}\equiv M^{-1}$ and 
\begin{equation}
\label{eqn:MTaylor}
\tilde{M}(y) = \tilde{M}(x) + \nabla \tilde{M}(x)(y-x) + r_2(x;y-x) + r_3(x;y-x) + O(\|y-x\|^4),
\end{equation}
where $\nabla \tilde{M}(x)(y-x)$ is a matrix with entries 
$$
\left(\nabla \tilde{M}(x)(y-x)\right)_{ij} = \nabla \tilde{M}_{ij}^\top(y-x),
$$
and $r_2(x;z)$ and $r_3(x;z)$ are the matrices whose entries are the second and third-order terms in Taylor expansions of $\tilde{M}_{ij}$.
Then
\begin{align}
\mathcal{G}_{\epsilon}f(x) &=\frac{(2\pi\epsilon)^{d/2}}{|\tilde{M}|^{1/2}}\left(f(x) +\epsilon\left[- \nabla f(x)^\top \omega_1(x) -f(x)\omega_2(x) \right.\right.\notag\\
&\left.\left.+\frac{1}{2}{\sf tr}(\tilde{M}(x)^{-1}H(x))\right] + O(\epsilon^2)\right), \label{eqn:Ge1}
\end{align}
where $H(x):=\nabla\nabla f(x)$ is the Hessian matrix for $f$ evaluated at $x$, 
and 
\begin{align}
\omega_{1,i}(x):= &
 \frac{|\tilde{M}|^{1/2}}{(2\pi\epsilon)^{d/2}} \frac{1}{4\epsilon^2}\int_{\R^d}e^{-\frac{z^\top \tilde{M} z}{2\epsilon} }z_i \left[z^\top [\nabla \tilde{M}(x){z}]z\right]dz,
 \quad 1\le i\le d,
  \label{eqn:o10}\\
\omega_2(x):=&
 \frac{|\tilde{M}|^{1/2}} {(2\pi\epsilon)^{d/2}} 
 \int_{\R^d} e^{-\frac{z^\top \tilde{M} z}{2\epsilon} } \left[ \frac{z^\top r_2(x;z)z}{4\epsilon^2} -
  \frac{\left({z}^\top [\nabla \tilde{M}(x){z} ]z\right)^2}{32\epsilon^3}\right]dz. \label{eqn:o20}
\end{align}
\end{lemma}
\begin{proof}
The proof utilizes the 
Taylor expansion $f(x+z)$ around $f(x)$:
\begin{equation}
\label{eqn:fTaylor}
f(x + z) = f(x) + \nabla f(x)^\top z + \frac{1}{2}z^\top\nabla\nabla f(x) z + p_3(z) + p_4(z),
\end{equation}
where $p_3(z)$ is a homogeneous third degree polynomial in $z$, and $p_4(z)$ is $O(z^4)$.
To establish \eqref{eqn:Ge1}, we will need to integrate the products of each of these terms with the Mahalanobis kernel 
\begin{equation}
\label{eqn:m1kernel}
k_{\epsilon}(x,y) = e^{-\frac{1}{4\epsilon}(x-y)^\top [\tilde{M}(x) +\tilde{M}(y)](x-y)}.
\end{equation}
First we eliminate the dependence of the matrix $\tilde{M}$ on the integration variable $y$ in the exponent by using Taylor expansions. 
Let  $z:=y-x$ and $\varrho$ be the residual in the expansion \eqref{eqn:MTaylor}:
$$
\varrho(x,x+z) : = \tilde{M}(x+z)-\left\{\tilde{M}(x) + \nabla \tilde{M}(x)z+ r_2(x;z) + r_3(x;z)\right\} = O(\|z\|^4).
$$
Then
\begin{equation}
\label{eqn:eaux}
e^{-\frac{(y-x)^\top [\tilde{M} (x) + \tilde{M}(y)](y-x)}{4\epsilon}}=
e^{-\frac{{z}^\top \tilde{M} (x){z}}{2\epsilon}} e^{-\frac{{z}^\top [\nabla \tilde{M}(x){z} + r_2(x;z) + r_3(x;z)+ \varrho(x,y)] {z}}{4\epsilon}}.
\end{equation}
The Taylor series for $\exp(-t)$ converges on $\R$.  Hence, 
expanding the second exponent in \eqref{eqn:eaux} we get:
\begin{align*}
&e^{-\frac{ 
{z}^\top [\nabla \tilde{M}(x){z} + r_2(z) + r_3(z)+\varrho(x,y)] z
}{4\epsilon}}\notag \\
=&
1 - \frac{1}{4\epsilon}
\left({z}^\top [\nabla \tilde{M}(x){z} + r_2(x;z) + r_3(x;z)+ \varrho(x,y)] {z}\right)\\
+&\frac{1}{32\epsilon^2}
\left({z}^\top [\nabla \tilde{M}(x){z} + r_2(x;z) + r_3(x;z)+ \varrho(x,y)] {z}\right)^2 -
\ldots\\
+&\frac{(-1)^k}{k!(4\epsilon)^k}
\left({z}^\top [\nabla \tilde{M}(x){z} + r_2(x;z) + r_3(x;z)+ \varrho(x,y)] {z}\right)^k +\ldots.
\end{align*}

Second, we will split the integral of each term in \eqref{eqn:fTaylor} multiplied by $k_{\epsilon}(x,y)$ into the sum
$$
\int_{\R^d} = \int_{\mathcal{B}_{x,\epsilon^{\mu}} } + \int_{\R^d\backslash\mathcal{B}_{x,\epsilon^{\mu}}},
$$
where $\mathcal{B}_{x,\epsilon^\mu}$
denotes the ellipse $\mathcal{B}_{\epsilon^{\mu}}(x;\tilde{M})$ and $\mu\in(0,\sfrac{1}{2})$ is fixed.
Where appropriate, we will apply Lemma \ref{thm:lem0} to the integral over $\R^d\backslash\mathcal{B}_{x,\epsilon^{\mu}}$. 
We will need the following ingredients.
 
{\bf Integral 1}:
\begin{align}
&\int_{\mathcal{B}_{x,\epsilon^{\mu}}}k_{\epsilon}(x,y)dy\label{eqn:int1}\\
=&
\int_{\mathcal{B}_{0,\epsilon^{\mu}}}e^{-\frac{z^\top \tilde{M}(x) z}{2\epsilon}}\left[1- \frac{1}{4\epsilon} \left({z}^\top [\nabla \tilde{M}(x){z} + r_2(x;z) + r_3(x;z)+ O(\|z\|^4)] {z}\right)\right.\notag\\
&+\left.\frac{1}{32\epsilon^2}
\left({z}^\top [\nabla \tilde{M}(x){z} + r_2(x;z) + r_3(x;z)+ \rho(x,y)] {z}\right)^2-
\ldots\right.\notag\\
+&\left.\frac{(-1)^k}{k!(4\epsilon)^k}
\left({z}^\top [\nabla \tilde{M}(x){z} + r_2(x;z) + r_3(x;z)+ \rho(x,y)] {z}\right)^k +\ldots
\right]dz.\notag
\end{align}
We will tackle this integral term-by-term. For brevity, we will omit the argument $(x)$ of $\tilde{M}$. Thus,
$$
\int_{\mathcal{B}_{0,\epsilon^{\mu}}}e^{-\frac{z^\top \tilde{M} z}{2\epsilon}}dz = \frac{(2\pi\epsilon)^{d/2}}{|\tilde{M}|^{1/2}}(1 + \delta_1(\epsilon)),
$$
 where $\delta_1(\epsilon)$ decays exponentially fast as $\epsilon\rightarrow 0$. Further,
\begin{align*}
\int_{\mathcal{B}_{0,\epsilon^{\mu}}}e^{-\frac{z^\top \tilde{M} z}{2\epsilon}} \frac{ \left({z}^\top [\nabla \tilde{M}{z} ]z\right)^{2k+1}}{(4\epsilon)^{2k+1}}dz &= 0,\quad k = 0,1,2,\ldots;\\
\int_{\mathcal{B}_{0,\epsilon^{\mu}}}e^{-\frac{z^\top \tilde{M} z}{2\epsilon}} \frac{ \left({z}^\top [\nabla \tilde{M}{z} ]z\right)^{2k}}{(4\epsilon)^{2k}}dz 
&=  \frac{(2\pi\epsilon)^{d/2}}{|\tilde{M}|^{1/2}} O\left(\epsilon^{k}\right)  ,
\quad k = 1,2,\ldots;\\
\int_{\mathcal{B}_{0,\epsilon^{\mu}}}e^{-\frac{z^\top \tilde{M} z}{2\epsilon}} \frac{\left({z}^\top r_2(x;z) z\right)^{k}}{(4\epsilon)^{k}} dz &= 
 \frac{(2\pi\epsilon)^{d/2}}{|\tilde{M}|^{1/2}} O\left(\epsilon^{2k-1}\right)  ,\quad k = 1,2,\ldots;\\
\int_{\mathcal{B}_{0,\epsilon^{\mu}}}e^{-\frac{z^\top \tilde{M} z}{2\epsilon}}
 \frac{\left({z}^\top [\nabla \tilde{M}{z} ]z\right)^{2k}{z}^\top r_2 (x;z)z}{(4\epsilon)^{2k+1}}dz &= 
 \frac{(2\pi\epsilon)^{d/2}}{|\tilde{M}|^{1/2}} O\left(\epsilon^{k+1}\right)  ,\quad k = 1,2,\ldots.
 \end{align*}
 The rest of the integrals originating from \eqref{eqn:int1} will be either zero or $O(\epsilon^{d/2}\epsilon^k)$ for some $k>2$.
Putting the integrals together and organizing them according to the order in  powers of $\epsilon$, we get:
\begin{equation}
\int_{\mathcal{B}_{x,\epsilon^{\mu}}}k_{\epsilon}(x,y)dy
= \frac{(2\pi\epsilon)^{d/2}}{|\tilde{M}|^{1/2}}\left[1 -\epsilon\omega_2(x) + O(\epsilon^2)\right], \label{I1}
\end{equation} 
where 
\begin{equation}
\label{eqn:o2}
\omega_2(x):= \frac{|\tilde{M}|^{1/2}}{(2\pi\epsilon)^{d/2}}
 \int_{\R^d} e^{-\frac{z^\top \tilde{M} z}{2\epsilon} } \left[ \frac{z^\top r_2(x;z)z}{4\epsilon^2} -
  \frac{\left({z}^\top [\nabla \tilde{M}{z} ]z\right)^2}{32\epsilon^3}\right]dz.
\end{equation}
Note that the value $\omega_2$ is of the order of 1. We have applied Lemma \ref{thm:lem0} to replace the integral over 
$\mathcal{B}_{x,\epsilon^{\mu}}$ with the one over $\R^d$. 


{\bf Integral 2}:
\begin{align}
&\int_{\mathcal{B}_{0,\epsilon^{\mu}}}k_{\epsilon}(x,x+z)z_idz\notag\\
=&
\int_{\mathcal{B}_{0,\epsilon^{\mu}}}e^{-\frac{z^\top \tilde{M} z}{2\epsilon}}z_i\left[1- \frac{1}{4\epsilon} \left({z}^\top [\nabla \tilde{M}{z} + r_2(z) + r_3(z)+ O(\|z\|^4)] {z}\right)+\ldots\right]dz\notag\\
=& \frac{(2\pi\epsilon)^{d/2}}{|\tilde{M}|^{1/2}}\left[ -\epsilon\omega_{1,i}(x) + O(\epsilon^2)\right], \label{I2}
\end{align} 
where, with the aid of Lemma \ref{thm:lem0},
\begin{equation}
\label{eqn:o1}
\omega_{1,i}(x):=
 \frac{|\tilde{M}|^{1/2}}{(2\pi\epsilon)^{d/2}}  \frac{1}{4\epsilon^2}\int_{\R^d}e^{-\frac{z^\top \tilde{M} z}{2\epsilon} }z_i \left[z^\top [\nabla \tilde{M}{z}]z\right]dz,\quad 1\le i\le d.
\end{equation}


{\bf Integral 3}:
\begin{align}
&\int_{\mathcal{B}_{0,\epsilon^{\mu}}}k_{\epsilon}(x,x+z)z^\top Hzdz\notag\\
=&
\int_{\mathcal{B}_{0,\epsilon^{\mu}}}e^{-\frac{z^\top \tilde{M} z}{2\epsilon}}z^\top H z\left[1- \frac{1}{4\epsilon} \left({z}^\top [\nabla \tilde{M}(x){z} + r_2(z) + r_3(z)+ O(\|z\|^4)] {z}\right)\right]dz\notag\\
=& \frac{(2\pi\epsilon)^{d/2}}{|\tilde{M}|^{1/2}}\left[ \epsilon{\sf tr}(\tilde{M}^{-1}H) + O(\epsilon^2)\right]. \label{I3}
\end{align} 

Using Integrals 1, 2, and 3, we calculate:
$$
\mathcal{G}_{\epsilon}f(x)  = \int_{\R^d} e^{-\frac{z^\top [\tilde{M}(x) + \tilde{M}(x+z)] z}{4\epsilon}}f(x+z) dz 
= \int_{\mathcal{B}_{0,\epsilon^{\mu}}} [\ldots]dz + \int_{\R^d\backslash \mathcal{B}_{0,\epsilon^{\mu}}}[\ldots]dz.
$$
The second integral in the right-hand side decays exponentially as $\epsilon\rightarrow 0$ by Lemma \ref{thm:lem0}. 
Therefore, we will incorporate its value into $O(\epsilon^2)$ term below. We continue, omitting the argument $x$ for brevity and recalling that $\tilde{M}^{-1}\equiv M$:
\begin{align*}
\mathcal{G}_{\epsilon}f& = 
 \int_{\mathcal{B}_{0,\epsilon^{\mu}}}  e^{-\frac{z^\top [\tilde{M}(x) + \tilde{M}(x+z)] z}{2\epsilon}}\left[f + \nabla f^\top z + \frac{1}{2}z^\top \nabla\nabla f z +\ldots\right]dz + \int_{\R^d\backslash \mathcal{B}_{0,\epsilon^{\mu}}}[\ldots]dz
\\
&=
(2\pi\epsilon)^{d/2}|M|^{1/2}\left( f +\epsilon\left[- \nabla f^\top \omega_1 -f\omega_2 +\frac{1}{2}{\sf tr}(M\nabla\nabla f)\right] + O(\epsilon^2)\right).
\end{align*}
\end{proof}

Now we prove Theorem \ref{thm:main-theorem}.
\begin{proof}
To carry out the proof of Theorem \ref{thm:main-theorem},  we need to calculate the limit
\begin{equation}
\label{eqn:limstar}
\lim_{\epsilon\rightarrow0}\mathcal{L}_{\epsilon,\alpha}f(x)\equiv \lim_{\epsilon\rightarrow0}\frac{\mathcal{P}_{\epsilon,\alpha}f(x) - f(x)}{\epsilon}
\end{equation}
for any fixed $x\in\mathcal{M}$.
Central to the calculation of $\mathcal{P}_{\epsilon,\alpha}f(x)$ is the calculation of integrals over $\mathcal{M}$ which is done 
by splitting each integral into the sum 
$$
\int_{\mathcal{M}} = \int_{\mathcal{B}_{x,\epsilon^{\mu}}}+ \int_{\mathcal{M}\backslash \mathcal{B}_{x,\epsilon^{\mu}}},
$$
where $\mu\in(0,\sfrac{1}{2})$ meaning that $\epsilon^{\mu}\rightarrow 0$ as $\epsilon\rightarrow 0$.
{ Since the manifold $\mathcal{M}$ is either $\R^d$ or $\mathbb{T}^k\times \mathbb{R}^{d-k}$ for some $1\le 1\le d$ with the Euclidean metric within any open ball of radius $R_{Euc}$ \eqref{eqn:Reuc}, for the purpose of integration over it, $\mathcal{M}$ can be treated either as $\mathbb{R}^d$ or as a hyperstrip or a hyperbox in $\mathbb{R}^d$ (See Assumption \ref{Ass1}).
Therefore, if $\epsilon$ is small enough so that the whole ellipse $\mathcal{B}_{x,\epsilon^{\mu}}$ lies within a ball of radius $R_{Euc}$, for each integral with an integrand satisfying the assumptions of Lemma \ref{thm:lem0} we have:
$$
\int_{\mathcal{M}}  = \int_{\mathcal{B}_{x,\epsilon^{\mu}}}+ \int_{{\mathcal{M}}\backslash \mathcal{B}_{x,\epsilon^{\mu}}}\quad{\rm and}\quad 
\left| \int_{{\mathcal{M}}\backslash \mathcal{B}_{x,\epsilon^{\mu}}}\right| \le \left| \int_{\R^d \backslash \mathcal{B}_{x,\epsilon^{\mu}}}\right|.
$$
}
According to Lemma \ref{thm:lem0}, this last integral over $\R^d \backslash \mathcal{B}_{x,\epsilon^{\mu}}$ decays exponentially as $\epsilon\rightarrow 0$.
Therefore, the integrals  over $\mathcal{M}\backslash \mathcal{B}_{x,\epsilon^{\mu}}$ do not affect the limit \eqref{eqn:limstar}, i.e.,
\eqref{eqn:limstar} is completely determined by the integrals over the ellipse $\mathcal{B}_{x,\epsilon^{\mu}}$ 
which are the same whether $\mathcal{M}$ is $\R^d$ or $\mathbb{T}^k\times \mathbb{R}^{d-k}$  provided that it satisfies Assumption \ref{Ass1}.
{ Hence, Lemma \ref{thm:lem2} remains valid for the manifold $\mathcal{M}$.}

We will omit the argument $x$ in the calculations within this proof to shorten expressions.
Lemma \ref{thm:lem2} implies that 
\begin{align}
\rho_{\epsilon}(x) &= \int_{\mathcal{M}} k_{\epsilon}(x,y)\rho(y) dy \notag \\
&=(2\pi\epsilon)^{d/2}|M|^{1/2}
\left( \rho +\epsilon\left[- \nabla \rho^\top \omega_1 -\rho\omega_2 +\frac{1}{2}{\sf tr}(M\nabla\nabla \rho)\right] + O(\epsilon^2)\right). \label{eq:c13}
\end{align}
Therefore,
\begin{equation}
\label{eqn:c14}
\rho_{\epsilon}^{-\alpha} =\frac{(2\pi\epsilon)^{-\alpha d/2} |\tilde{M}|^{\alpha /2}}{\rho^{\alpha}}
\left[ 1 - \alpha \epsilon \frac{- \nabla \rho^\top \omega_1 -\rho\omega_2 +\frac{1}{2}{\sf tr}(M\nabla\nabla \rho)}{\rho} + O(\epsilon^2)\right].
\end{equation}
To perform right renormalization, we need to multiply the kernel $k_{\epsilon}(x,y)$ by $\rho_{\epsilon}^{-\alpha}(y)$.
The dependence of $k_{\epsilon}(x,y)\rho_{\epsilon}^{-\alpha}(y)$ on $y$ will be shifted to terms of Taylor expansions. As before, let $z = y-x$. 
First, we expand $|\tilde{M}(y)|^{\alpha/2}$
\begin{align}
&|\tilde{M}(x+z)|^{\alpha/2} =\left[|\tilde{M}(x)| +\nabla|\tilde{M}(x)|^\top z+ \frac{1}{2}z^\top \nabla\nabla|\tilde{M}(x)|z +O(\|z\|^3)\right]^{\alpha/2}\notag \\
=&  |\tilde{M}(x)|^{\alpha/2}\left[1 +\frac{\nabla|\tilde{M}(x)|^\top z}{|\tilde{M}(x)|}+\frac{z^\top \nabla\nabla|\tilde{M}(x)|z}{2|\tilde{M}(x)|} +O(\|z\|^3)\right]^{\alpha/2}\notag\\
=&|\tilde{M}(x)|^{\alpha/2}\left[1 +\frac{\alpha}{2}\frac{\nabla |\tilde{M}(x)|^\top z}{|\tilde{M}(x)|} + \frac{1}{2}z^\top A_1(x)z +  O(\|z\|^3)\right] \label{eqn:qmaTaylor}
\end{align}
where 
$$
A_1(x) = \frac{\alpha \nabla \nabla |\tilde{M}(x)|)}{2|\tilde{M}(x)| } +  \frac{\alpha(\alpha-2)\nabla |\tilde{M}(x)|\nabla |\tilde{M}(x)|^\top}{4|\tilde{M}(x)|^2}.
$$
Second, we denote the term multiplied by $\alpha\epsilon$ in \eqref{eqn:c14} by $R$:
$$
R: =  \frac{- \nabla \rho^\top \omega_1 -\rho\omega_2 +\frac{1}{2}{\sf tr}(M\nabla\nabla \rho)}{\rho}.
$$
Now, using \eqref{eqn:qmaTaylor},  we start the calculation of $\mathcal{K}_{\epsilon,\alpha}f(x)$:
\begin{align}
\mathcal{K}_{\epsilon,\alpha}f(x) &:= \int_{\mathcal{M}}\frac{ k_{\epsilon}(x,y)}{\rho^{\alpha}_{\epsilon}(y)}\rho(y) f(y)dy \notag\\
& = \int_{\mathcal{M}}k_{\epsilon}(x,y)\left[\rho^{-\alpha}_{\epsilon}(y)\rho(y)f(y)\right]dy \notag\\
& =\left[ \frac{|\tilde{M}(x)|}{(2\pi\epsilon)^d} \right]^{\alpha/2} \int_{\mathcal{M}}k_{\epsilon}(x,y)\left[\rho^{1-\alpha}(y)f(y)\right]\label{eqn:Kea1} \\
&\times
\left[ 1 - \alpha \epsilon R(y)+ O(\epsilon^2)\right]\notag\\
&\times
\left[1 +\frac{\alpha}{2}\frac{\nabla|\tilde{M}(x)|^\top (y-x)}{|\tilde{M}(x)|}+ \frac{1}{2}(y-x)^\top A_1(x)(y-x) +  O(\|y-x\|^3)\right] dy.\notag
\end{align}
To tackle the integral in \eqref{eqn:Kea1}, we split it to several integrals each of which we evaluate using Lemma \ref{thm:lem2}. 
For  brevity, we will omit arguments $x$ in the gradients and Hessians and in the matrix $\tilde{M}$. 
We continue:
\begin{align}
&
\int_{\mathcal{M}}k_{\epsilon}(x,y)\left[\rho^{1-\alpha}(y)f(y) \right]
\left[ 1 - \alpha \epsilon R(y)+ O(\epsilon^2)\right]\notag\\
&\times
\left[1 +\frac{\alpha}{2}\frac{\nabla|\tilde{M}|^\top z}{|\tilde{M}|}+ \frac{1}{2}z^\top A_1z+  O(\|z\|^3)\right] dz\notag\\
=&
\mathcal{G}_{\epsilon}[\rho^{1-\alpha}(y)f(y)](x) -\alpha\epsilon \mathcal{G}_{\epsilon}[ \rho^{1-\alpha}(y)f(y)R(y)](x)\notag\\
+&
\frac{\alpha}{2}\mathcal{G}_{\epsilon}\left[\rho^{1-\alpha}(y)f(y)\frac{\nabla|\tilde{M}|^\top (y-x)}{|\tilde{M}|}\right] (x)\notag\\
+& 
\frac{1}{2} \mathcal{G}_{\epsilon}\left[\rho^{1-\alpha}(y)f(y)(y-x)^\top A_1(y-x)\right](x) +  \frac{(2\pi\epsilon)^{d/2}}{|\tilde{M}|^{1/2}}O(\epsilon^2). \label{eqn:Gea2a}
\end{align}
Applying Lemma \ref{thm:lem2} we compute the four operators $\mathcal{G}_{\epsilon}$ in the last equation:
\begin{align}
\mathcal{G}_{\epsilon}[\rho^{1-\alpha}f] &=  \frac{(2\pi\epsilon)^{d/2}}{|\tilde{M}|^{1/2}}
\left[\rho^{1-\alpha}f\left[1-\epsilon\omega_2(x)\right]\right.\notag\\
&\left.+\epsilon \left\{ -\nabla(f\rho^{1-\alpha})^\top\omega_1 +
 \frac{1}{2} 
{\sf tr} (M\nabla\nabla \left[\rho^{1-\alpha}f\right]) \right\}  +O(\epsilon^2)\right];
\label{eqn:Gea1}\\
\alpha\epsilon\mathcal{G}_{\epsilon}[ \rho^{1-\alpha}fR] & = \frac{(2\pi\epsilon)^{d/2}}{|\tilde{M}|^{1/2}}\left[\alpha\epsilon \rho^{1-\alpha}fR + O(\epsilon^2)\right]; \label{eqn:Gea2}
\end{align}
\begin{align}
& \mathcal{G}_{\epsilon} \left[\rho^{1-\alpha}(y)f(y)\frac{\nabla|\tilde{M}|^\top (y-x)}{|\tilde{M}|}\right](x) =
\frac{(2\pi\epsilon)^{d/2}}{|\tilde{M}|^{1/2}}
\left\{-\epsilon\frac{\rho^{1-\alpha}f\nabla|\tilde{M}|^\top\omega_1}{|\tilde{M}|} \right.\notag\\
+&\left.
 \frac{\epsilon}{2}
 {\sf tr}\left(M\nabla_y\nabla_y\left[\rho^{1-\alpha}(y)f(y)\frac{\nabla|\tilde{M}|^\top (y-x)}{|\tilde{M}|}\right]_{y=x} + O(\epsilon^2)\right)\right\}.
 \label{eqn:Gea3}
\end{align}
Let us calculate the Hessian matrix in \eqref{eqn:Gea3}:
\begin{align}
 \nabla_y\nabla_y\left[\rho^{1-\alpha}(y)f(y)\frac{\nabla|\tilde{M}|^\top (y-x)}{|\tilde{M}|}\right]_{y=x} & = 
\nabla_y\left[\nabla_y\left[\rho^{1-\alpha}(y)f(y)\right] \frac{\nabla|\tilde{M}|^\top (y-x)}{|\tilde{M}|}\right.\notag\\
+  \left. 
\rho^{1-\alpha}(y)f(y)\nabla_y\frac{\nabla|\tilde{M}|^\top (y-x)}{|\tilde{M}|} \right]_{y=x} &= 2\nabla\left[\rho^{1-\alpha}f\right] \frac{\nabla|\tilde{M}|^\top}{|\tilde{M}|}.
\label{yyterm}
\end{align}
Finally, we compute the operator $\mathcal{G}_{\epsilon}$ in \eqref{eqn:Gea2a}:
\begin{equation}
\label{eqn:Gea4}
\mathcal{G}_{\epsilon}\left[\rho^{1-\alpha}(y)f(y)(y-x)^\top A_1(y-x)\right](x) = \frac{(2\pi\epsilon)^{d/2}}{|\tilde{M}|^{1/2}}
\left[\frac{\epsilon}{2} \rho^{1-\alpha } f{\sf tr}(M A_1) + O(\epsilon^2)\right].
\end{equation}
Putting together \eqref{eqn:Kea1}--\eqref{eqn:Gea4} we obtain:
\begin{align}
&\mathcal{K}_{\epsilon,\alpha}f(x) = 
\left[ \frac{|\tilde{M}(x)|}{(2\pi\epsilon)^d} \right]^{\frac{\alpha-1}{2}} 
\left[ \rho^{1-\alpha}f \left[1 +\epsilon h(x)\right] \right. \notag\\
+ &\epsilon \left\{-\nabla(f\rho^{1-\alpha})^\top\omega_1
+ \frac{1}{2} {\sf tr} (M\nabla\nabla \left[\rho^{1-\alpha}f\right]) +
\frac{\alpha}{2} {\sf tr} \left(M\nabla\left[\rho^{1-\alpha}f\right] \frac{\nabla|\tilde{M}|^\top}{|\tilde{M}|}\right) \right\}\notag\\
+&\left.O(\epsilon^2)\right],\label{eqn:Kea2}
\end{align}
where 
$$
h(x) = -\omega_2(x) -\alpha R(x)-\frac{\alpha}{2}\frac{\nabla|\tilde{M}|^\top\omega_1}{|\tilde{M}|} + \frac{1}{4}{\sf tr}(MA_1).
$$ 
 To facilitate the calculation, we denote the expression in the curly brackets in \eqref{eqn:Kea2} divided by $q^{1-\alpha}$ by
 $$
 B(\rho,f): =  \left[ \frac{- \nabla(f\rho^{1-\alpha})^\top\omega_1 +
\frac{1}{2} 
{\sf tr} (M\nabla\nabla \left[\rho^{1-\alpha}f\right]) +
\frac{\alpha}{2} {\sf tr} \left(M\nabla\left[\rho^{1-\alpha}f\right] \frac{\nabla|\tilde{M}|^\top}{|\tilde{M}|}\right)
}
{\rho^{1-\alpha}}\right] .
$$
Then $\mathcal{K}_{\epsilon,\alpha}f(x)$ can be written as:
\begin{equation}
\label{eqn:Kea3}
\mathcal{K}_{\epsilon,\alpha}f(x) = 
\left[ \frac{|\tilde{M}(x)|}{(2\pi\epsilon)^d} \right]^{\frac{\alpha-1}{2}}\rho^{1-\alpha}
\left[ f \left[1 +\epsilon h(x)\right] 
+\epsilon B(\rho,f) +O(\epsilon^2)\right].
\end{equation}

Observing that $\rho_{\epsilon,\alpha}(x) = K_{\epsilon,\alpha}1$, i.e., we need to use $f\equiv 1$ to get $\rho_{\epsilon,\alpha}(x)$, we calculate
the operator $\mathcal{P}_{\epsilon,\alpha}$:
\begin{align}
\mathcal{P}_{\epsilon,\alpha}f(x) & = \int_{\mathcal{M}}\frac{k_{\epsilon,\alpha}(x,y)}{\rho_{\epsilon,\alpha}(x)}  f(y) \rho(y) dy \notag \\
& =
 \frac{f\left\{1+\epsilon h(x)\right\} + \epsilon B(\rho,f) + O(\epsilon^2)}
{1+\epsilon h(x) + \epsilon B(\rho,1) + O(\epsilon^2)}. \label{eqn:c15anew}
 \end{align}
Expanding $\mathcal{P}_{\epsilon,\alpha}f(x) $ in powers of $\epsilon$ we obtain:
\begin{align}
\mathcal{P}_{\epsilon,\alpha}f(x) & =
\left[
f\left(1+\epsilon{h}(x)\right) + \epsilon B(\rho,f) + O(\epsilon^2)
\right]
\left[1-\epsilon{h}(x) - \epsilon B(\rho,1) + O(\epsilon^2)\right] 
\notag \\
 &=f +\epsilon\left[ B(\rho,f) - f B(\rho,1)\right]\notag\\
 &=
 f +\epsilon  \frac{- \nabla(f\rho^{1-\alpha})^\top\omega_1 +
\frac{1}{2} 
{\sf tr} (\tilde{M}^{-1}\nabla\nabla \left[\rho^{1-\alpha}f\right]) + \frac{\alpha}{2} {\sf tr} \left(\tilde{M}^{-1}\nabla\left[\rho^{1-\alpha}f\right] \frac{\nabla|\tilde{M}|^\top}{|\tilde{M}|}\right)
}
{\rho^{1-\alpha}} \notag\\
&- \epsilon f\frac{- \nabla(\rho^{1-\alpha})^\top\omega_1 +
\frac{1}{2} 
{\sf tr} (\tilde{M}^{-1}\nabla\nabla \left[\rho^{1-\alpha}\right]) + 
\frac{\alpha}{2} {\sf tr} \left(\tilde{M}^{-1}\nabla\left[\rho^{1-\alpha}\right] \frac{\nabla|\tilde{M}|^\top}{|\tilde{M}|}\right)
}
{\rho^{1-\alpha}}
\label{eqn:c16}
\end{align}
Finally, we compute the operator $\mathcal{L}_{\epsilon,\alpha}$, take the limit $\epsilon\rightarrow 0$, and obtain the desired result:
\begin{align} 
\lim_{\epsilon\rightarrow 0} \frac{\mathcal{P}_{\epsilon,\alpha}f(x) -f(x)}{\epsilon} &= 
\frac{1}{2}\left(
\frac{{\sf tr}
\left(M
\left[ \nabla\nabla \left[\rho^{1-\alpha}f\right] - f\nabla\nabla \rho^{1-\alpha}\right]   
\right)
}{\rho^{1-\alpha}} 
\right) 
\notag\\
&+
\frac{\alpha}{2}\left(
\frac{{\sf tr}
\left(M
\left[ 
\nabla\left[\rho^{1-\alpha}f\right]  - 
f\nabla\left[\rho^{1-\alpha}\right] 
\right]   \frac{\nabla|\tilde{M}|^\top}{|\tilde{M}|}
\right)
}{\rho^{1-\alpha}} 
 \right) \notag \\
 &-\left(\frac{ [\nabla(f\rho^{1-\alpha})-f\nabla(\rho^{1-\alpha})]^\top \omega_1]}{\rho^{1-\alpha}}
\right).
\label{eqn:Lea1}
\end{align}
\end{proof}

%
%
\section{Proof of Corollary  \ref{thm:main_corollary}}
\label{appendix:alpha0.5}
\begin{proof}
{\bf Term 1 in \eqref{eqn:Lea1}.}
First we compute
\begin{equation}
\label{eqn:Le05}
 \frac{1}{2}\left(\frac{
{\sf tr} (M\nabla\nabla \left[\rho^{1/2}f\right])
}{\rho^{1/2}}  -
 f\frac{
{\sf tr} (M\nabla\nabla \left[\rho^{1/2}\right])
}{\rho^{1/2}}  \right) .
\end{equation}
Since $\rho$ is the Gibbs density, we have:
$$
\rho = Z^{-1}e^{-\beta F}.\quad{\rm Hence}\quad  \rho^{1/2} = Z^{-1/2}e^{-\tfrac{\beta}{2} F},\quad \nabla \rho^{1/2} = -\left[\frac{\beta}{2}\nabla F\right] \rho^{1/2}.
$$
We will use the fact that
$$
\nabla\nabla[ \rho^{1/2} f] = f\nabla\nabla \rho^{1/2} + [\nabla \rho^{1/2}(\nabla f)^\top + \nabla f(\nabla \rho^{1/2})^\top] + \rho^{1/2}\nabla\nabla f.
$$
Applying  the property ${\sf tr}(AB) = {\sf tr}(BA)$ and recalling that $M$ is symmetric, we obtain:
\begin{align*}
{\sf tr}[M[\nabla \rho^{1/2}(\nabla f)^\top + \nabla f(\nabla \rho^{1/2})^\top]] =& 
{\sf tr}[M\nabla \rho^{1/2}\nabla f^\top] +
 {\sf tr}[M^{\top}\nabla f \nabla \rho^{1/2}]  
 \\
 =&
   {\sf tr}[M\nabla \rho^{1/2}\nabla f^\top]  +   {\sf tr}[\nabla f (\nabla \rho^{1/2})^\top M^{\top}]  \\
 =&
      {\sf tr}[M\nabla \rho^{1/2}\nabla f^\top] +{\sf tr}[\nabla f (M\nabla \rho^{1/2})^\top]  \\
 =&
      {\sf tr}[M\nabla \rho^{1/2}\nabla f^\top] +{\sf tr}[(M\nabla \rho^{1/2})\nabla f ^\top]  \\
  =&2     {\sf tr}[M\nabla \rho^{1/2}\nabla f^\top] \\
  =&2(M\nabla \rho^{1/2})^\top\nabla f.
\end{align*}
This yields:
$$
 \frac{1}{2}\left(\frac{
{\sf tr} (M\nabla\nabla \left[\rho^{1/2}f\right])
}{\rho^{1/2}}  -
 f\frac{
{\sf tr} (M\nabla\nabla \left[\rho^{1/2}\right])
}{\rho^{1/2}}  \right) 
=
\frac{1}{2}\left({\sf tr}[M\nabla\nabla f] -\beta (M\nabla F)^\top\nabla f \right).
$$
Note that this is $\sfrac{\beta}{2}$ times the generator \eqref{eqn:Lgen} for the case where the diffusion matrix $M$ is constant. 

{\bf Term 2 in \eqref{eqn:Lea1}.}
Utilizing the fact that
$$
\frac{ [\nabla(f\rho^{1/2})-f\nabla(\rho^{1/2})]}{\rho^{1/2}} =\nabla f,
$$
we simplify the second term in \eqref{eqn:Lea1}:
$$
\frac{\alpha}{2}\left(
\frac{{\sf tr}
\left(M
\left[ 
\nabla\left[\rho^{1-\alpha}f\right]  - 
f\nabla\left[\rho^{1-\alpha}\right] 
\right]   \frac{\nabla|\tilde{M}|^\top}{|\tilde{M}|}
\right)
}{\rho^{1-\alpha}} 
 \right) = 
 \frac{\alpha}{2}\left({\sf tr}\left[M\nabla f \right] \frac{\nabla|\tilde{M}|^\top}{|\tilde{M}|}
\right).
 $$
Let us introduce the notation
\begin{equation}
\label{eqn:RR}
R_k:= \tilde{M}^{-1/2} \frac{\partial \tilde{M}}{\partial x_k}\tilde{M}^{-1/2}
\end{equation}
It follows from Jacobi's formula for the derivative of the determinant that 
\begin{equation}
\label{eqn:Rk}
\frac{1}{|\tilde{M}|}\frac{\partial  |\tilde{M}|}{\partial x_k} = {\sf tr}\left(\tilde{M}^{-1}\frac{\partial \tilde{M}}{\partial x_k}\right) \equiv
 {\sf tr}\left(\tilde{M}^{-1/2}\frac{\partial \tilde{M}}{\partial x_k} \tilde{M}^{-1/2}\right)
 \equiv{\sf tr} R_k.
\end{equation}
Hence, we obtain:
\begin{equation}
\label{eqn:term2}
\frac{1}{4}{\sf tr}\left(M\nabla f [{\sf tr} R_1,\ldots,{\sf tr} R_d]\right)
=
\frac{1}{4}\sum_{k=1}^d\sum_{i=1}^dM_{ki}\frac{\partial f}{\partial x_i}{\sf tr}R_k.
\end{equation}

{\bf Term 3 in \eqref{eqn:Lea1}.}
Finally, we compute
\begin{align}
&\frac{ [\nabla(f\rho^{1/2})-f\nabla(\rho^{1/2})]^\top \omega_1}{\rho^{1/2}} =\nabla f^\top\omega_1,\quad{\rm where}\label{eqn:term3} \\
 \frac{(2\pi\epsilon)^{d/2}}{|\tilde{M}|^{1/2}}\omega_{1,i}(x) &=
  \frac{1}{4\epsilon^2}\int_{\mathcal{B}_{0,\sqrt{\epsilon}}}e^{-\frac{z^\top \tilde{M} z}{2\epsilon} }z_i \left[z^\top [\nabla \tilde{M}(x){z}]z\right]dz,\quad 1\le i\le d.\label{eqn:o11}
\end{align}
To compute the integral in \eqref{eqn:o11}, we do the variable change $t:=\epsilon^{-1/2}\tilde{M}^{1/2}z$. Then 
$$
z_i:=\epsilon^{1/2} e_i^\top \tilde{M}^{-1/2}t,\quad{\rm  where}\quad e_i \quad\text{is the standard unit vector}.
$$
The polynomial in the integrand in \eqref{eqn:o11} resulting from this change is:
\begin{align}
z_i \left[z^\top [\nabla \tilde{M}(x){z}]z \right]&=z_i z^\top\left[\sum_{k=1}^d \frac{\partial \tilde{M}}{\partial x_k}z_k\right]z 
 = \sum_{k=1}^d z_iz^\top \frac{\partial \tilde{M}}{\partial x_k} z z_k\notag \\
& = \epsilon^2 \sum_{k=1}^d e_i^\top \tilde{M}^{-1/2}t t^\top \tilde{M}^{-1/2} \frac{\partial \tilde{M}}{\partial x_k}\tilde{M}^{-1/2} tt^\top \tilde{M}^{-1/2} e_k. \label{z31}
\end{align}
Using the notation $R_k$ introduced in \eqref{eqn:RR} we get: 
\begin{equation}
\label{eqn:o12}
\omega_{1,i}(x) = \frac{1}{4(2\pi)^{d/2}} \sum_{k=1}^d e_i^\top \tilde{M}^{-1/2}\left[
 \int_{\R^d}e^{-\frac{t^2}{2} }t t^\top R_k tt^\top dt 
\right]
 \tilde{M}^{-1/2} e_k.
\end{equation}
First we compute the integral in the square brackets in \eqref{eqn:o12}.
This integral is a $d\times d$ matrix, and its entries are:
$$
(t t^\top R_k tt^\top)_{ij}  = \sum_{l=1}^d\sum_{m=1}^dt_it_l[R_k]_{lm}t_mt_j.
$$
{\bf Case $i=j$:}
Taking into account that in order to produce a nonzero integral, we must have $l=m$ in this case. Hence
\begin{align}
 \int_{\R^d}e^{-\frac{t^2}{2} }(t t^\top R_k tt^\top)_{ii} dt &= 
  \int_{\R^d}e^{-\frac{t^2}{2} } \sum_{l=1}^d[R_k]_{ll}t_l^2t_i^2dt \notag\\
  & =(2\pi)^{d/2}\left(3[R_k]_{ii} + \sum_{l\neq i}R_{ll}\right) =(2\pi)^{d/2}\ \left(2[R_k]_{ii} + {\sf tr}R_k\right).
\label{R1}
\end{align}

{\bf Case $i\neq j$:}
In this case, to produce a nonzero integral, we must have $l=i$ and $m = j$ or the other way around. Hence
\begin{align}
 \int_{\R^d}e^{-\frac{t^2}{2} }(t t^\top R_k tt^\top)_{ii} dt  & = 
  \int_{\R^d}e^{-\frac{t^2}{2} } \left[[R_k]_{ij}+[R_k]_{ji}\right]t_i^2t_j^2dt  \notag\\
  &= (2\pi)^{d/2}\left([R_k]_{ij}+ [R_k]_{ji}\right)= (2\pi)^{d/2}2[R_k]_{ij}\label{R2}
\end{align}
as $R_k$ is symmetric.

Therefore, the integral in the square brackets in \eqref{eqn:o12} is
\begin{equation}
\label{eqn:J}
(2\pi)^{d/2} \left(2R_k + I{\sf tr}R_k\right).
\end{equation}
Plugging this result into \eqref{eqn:o12}  and recalling \eqref{eqn:RR} we obtain:
\begin{align}
\omega_{1,i}(x) &=
\frac{1}{4} \sum_{k=1}^d e_i^\top \tilde{M}^{-1/2}
\left[
2R_k + I{\sf tr}R_k
\right]
 \tilde{M}^{-1/2} e_k\notag \\
 & =  \frac{1}{2}\sum_{k=1}^d e_i^\top \tilde{M}^{-1}\frac{\partial \tilde{M}}{\partial x_k} \tilde{M}^{-1}e_k + 
\frac{1}{4} \sum_{k=1}^de_i^\top \tilde{M}^{-1} {\sf tr}R e_k. \label{eqn:o13}
 \end{align}
Now we recall that 
$$
\frac{\partial \tilde{M}^{-1}}{\partial x_k} = -  \tilde{M}^{-1}\frac{\partial \tilde{M}}{\partial x_k}\tilde{M}^{-1}.
$$
Using this formula, we get:
\begin{equation}
\label{eqn:o14}
\omega_{1,i}(x)  =  -\frac{1}{2}\sum_{k=1}^d \frac{\partial M_{ik}}{\partial x_k}  +
 \frac{ 1 }{4} \sum_{k=1}^dM_{ik}{\sf tr}R_k.
 \end{equation}
 Therefore,
 \begin{equation}
 \label{eqn:term3}
 \nabla f^\top\omega_1 = -\frac{1}{2}\sum_{i=1}^d\sum_{k=1}^d \frac{\partial M_{ik}}{\partial x_k} \frac{\partial f}{\partial x_i}
 + \frac{ 1 }{4} \sum_{i=1}^d\sum_{k=1}^dM_{ik}{\sf tr}R_k\frac{\partial f}{\partial x_i}.
 \end{equation}

{\bf Getting the final result.}
Finally, we plug the calculated terms
in \eqref{eqn:Lea1}. We also use the fact that $M$ is symmetric, i.e. $M_{ik} = M_{ki}$ for $1\le i,k\le d$. We get:
\begin{align}
\lim_{\epsilon\rightarrow 0} L_{\epsilon,\alpha} & = 
\frac{1}{2}\left({\sf tr}[M\nabla\nabla f] -\beta (M\nabla F)^\top\nabla f \right)\notag \\
&+\frac{1}{4}\sum_{k=1}^d\sum_{i=1}^dM_{ki}\frac{\partial f}{\partial x_i}{\sf tr}R_k\notag\\
&+ \frac{1}{2}\sum_{i=1}^d\sum_{k=1}^d \frac{\partial M_{ik}}{\partial x_k} \frac{\partial f}{\partial x_i}
 -\frac{ 1 }{4} \sum_{i=1}^d\sum_{k=1}^dM_{ik}{\sf tr}R_k\frac{\partial f}{\partial x_i}\notag \\
 & = \frac{1}{2}\left({\sf tr}[M\nabla\nabla f] -\beta (M\nabla F)^\top\nabla f \right)
 +\frac{1}{2} \left(\nabla\cdot M\right)^\top \nabla f.
 \label{Lea2}
 \end{align}
The last expression is the generator $\mathcal{L}$ for the dynamics in collective variables \eqref{eqn:Lgen} multiplied by $\sfrac{\beta}{2}$ as desired.

\end{proof}


\section{Obtaining the reactive current and the reaction rate from the committor computed by {\tt mmap}}
\label{appendix:current}
We have computed the reactive current based on the Gamma operator
defined for an Ito diffusion with generator $\mathcal{L}$ as
\begin{equation}
    \Gamma(f,g)(x) = \frac{1}{2}\big(\mathcal{L}(fg)(x) - f\mathcal{L}g(x) - g\mathcal{L}f(x)).
\end{equation}
This operator is sometimes referred to as the 
\emph{carr\'{e} du champ}
operator~\cite{pavliotis2014stochastic,bakry2013analysis}. 
We apply it to the discrete generator matrix
$L$ from {\tt mmap}.

Applying the Gamma operator to the generator $\mathcal{L}$ of~\eqref{eqn:mgen} gives
\begin{equation}
    \label{eqn:chain_rule}
   \mathcal{L}(fg) = f \mathcal{L} g + g \mathcal{L}f + 2\beta^{-1} \nabla f^{\top}M \nabla g,
\end{equation}
and hence $\Gamma(f, g)$ simplifies to 
\begin{equation}
     \Gamma(f,g)(x) = \beta^{-1}\nabla f^{\top} M \nabla g(x).
\end{equation}
Choosing $f(x)$ to be the committor $q(x)$ and $g(x)$ to be $\chi_{\nu}:\R^d\rightarrow\R$, mapping $x$ to its $\nu$th component, 
we obtain the $\nu$th component of the reactive current by:
\begin{equation}
    \label{eqn:coord_carre}
    \rho(x)\Gamma(q,\chi_{\nu})(x) = \beta^{-1}\rho(x)[M\nabla q(x)]_{\nu}.
\end{equation}
Therefore, in order to obtain the reactive current
discretized to a dataset, we need to construct a discrete counterpart of the Gamma operator and obtain an estimate for the density $\rho$.

We recall that the discrete generator $L$ approximates $\frac{\beta}{2}\mathcal{L}$ pointwise on a dataset $\{x_i\}_{i=1}^{n}.$
Let $f$ and $g$ be arbitrary smooth functions and $[f], [g] \in \R^{n}$ be their discretization to the dataset, i.e., $[f]_i = f(x_i)$ and $[g]_i=g(x_i)$.
Since $L$ approximates $\frac{\beta}{2}\mathcal{L}$ pointwise on the dataset, we have:
\begin{equation}
\label{eq:D-4}
\sum\limits_{j} L_{ij}([f]_j [g]_j) - [f]_i L_{ij} [g]_j -
[g]_i L_{ij} [f]_j \approx \beta \Gamma(f, g)(x_i) =  \nabla
f^{\top} M \nabla g(x_i).
\end{equation}
Since the row sums of the matrix $L$  are zeros,  the left-hand side of \eqref{eq:D-4} can be written as
 \begin{equation}
\sum\limits_{j} L_{ij}([f]_j [g]_j) - [f]_i L_{ij} [g]_j - [g]_i L_{ij} [f]_j = \sum\limits_{j} L_{ij}([f]_i - [f]_j) ([g]_i - [g]_j).
\end{equation}
This allows us to define the discrete analogue of the Gamma operator by:
\begin{equation}
[\hat{\Gamma}(f, g)]_i := \beta^{-1}\sum\limits_{j=1}^n L_{ij}([f]_i - [f]_j) ([g]_i - [g]_j).
\end{equation}

Now, it remains to obtain an estimate for the density $\rho$. We proceed as follows.
First we construct an isotropic Gaussian kernel 
$[\tilde{k}_{\tilde{\epsilon}}]_{ij} = \exp[-||x_i - x_j||^2/(2\tilde{\epsilon})]$.
Setting $M(x) \equiv I$ and $f(x)\equiv 1$ in the kernel expansion~\eqref{eqn:Ge1}, we observe that
\begin{equation}
\lim\limits_{n\to \infty} \frac{1}{n(2\pi \tilde{\epsilon})^{d/2}} \sum\limits_{j=1}^n [\tilde{k}_{\tilde{\epsilon}}]_{ij} = \rho(x) + \mathcal{O}(\tilde{\epsilon}).
\end{equation}
So, we define a kernel density estimate with the vector $[p] \in \R^{n}$ defined by
\begin{equation}
[p]_i= 
 \frac{1}{n(2\pi \tilde{\epsilon})^{d/2}} \sum\limits_{j=1}^n [\tilde{k}_{\tilde{\epsilon}}]_{ij} \qquad i = 1,\ldots, n.
 \end{equation}
Alternatively, we can use the Mahalanobis kernel from {\tt mmap} with entries $[k_{\epsilon}]_{ij}$ and utilize~\eqref{eqn:Ge1} to define $[p]_i:= (n(2\pi \epsilon)^{d/2}|M(x_i)|^{1/2})^{-1} \sum_j [k_{\epsilon}]_{ij}$ for $i = 1,\ldots, n.$

Let  $[q] \in \R^{n}$ be the discrete committor obtained by {\tt mmap}.
Using the constructed discrete density $[p]$,  we estimate the reactive current using the formula
\begin{equation}
    \label{eqn:discrete_current}
    [\hat{\mathcal{J}}]_{\nu i} := \big[p\hat{\Gamma}\big([q], x^{\nu}\big)\big]_i= \beta^{-1}[p]_i \sum\limits_{j=1}^n L_{ij}([q]_i - [q]_j)(x_{i}^{\nu} -
    x_{j}^{\nu})
\end{equation}
where $1\le\nu\le d$, $1\le i\le n,$ and
$x_{i}^{\nu},$ 
    $x_{j}^{\nu}$ denote the $\nu$-th coordinate for data points $i$ and $j$ respectively.

The reaction rate $\nu_{AB}$ is given by~\eqref{eqn:nuAB} and can be rewritten as~\cite{EVE2010}
\begin{equation}
\nu_{AB} = \beta^{-1} \int_{\mathcal{M} \backslash (A \cup B)} \nabla q(x)^{\top} M(x) \nabla q(x) \rho(x) dx.
\end{equation}
Observing that $\beta^{-1} \nabla q(x)^{\top} M(x) \nabla q(x) \equiv \Gamma(q,q)$, we get:
\begin{equation}
    \label{eqn:rate_gamma}
\nu_{AB} = \int_{\mathcal{M} \backslash (A \cup B)} \Gamma(q, q)(x) \rho(x)dx.
\end{equation}
Hence, we compute an estimate $\hat{\nu}_{AB}$ as the Monte Carlo integral
\begin{equation}
    \label{eqn:rate_gamma}
\hat{\nu}_{AB} = \frac{1}{|I_{AB}|} \sum\limits_{i \in I_{AB}} \big[\hat{\Gamma}\big([q], [q]\big)\big]_i = \frac{1}{|I_{AB}|} \sum\limits_{i \in I_{AB}}\sum\limits_{j=1}^n   L_{ij}([q]_i - [q]_j)^2,
\end{equation}
where $I_{AB} = \{i: x_i \in \mathcal{M}
\backslash (A \cup B)\}$ .
\bibliographystyle{abbrv} 


\end{document}